\documentclass[10pt,a4paper]{amsart}
\usepackage[dvipdfmx]{graphicx}
\usepackage[dvipdfmx]{color}
\usepackage{marginnote}
\usepackage{amssymb}
\usepackage{ascmac}
\usepackage{enumitem} 
\usepackage{enumerate}
\usepackage{hyperref}
\usepackage{mathtools, stmaryrd}
\usepackage{xparse} 
\DeclarePairedDelimiterX{\Iintv}[1]{\llbracket}{\rrbracket}{\iintvargs{#1}}
\NewDocumentCommand{\iintvargs}{>{\SplitArgument{1}{,}}m}
{\iintvargsaux#1} %
\NewDocumentCommand{\iintvargsaux}{mm} {#1\mkern1.5mu,\mkern1.5mu#2}

\usepackage[normalem]{ulem}

\makeatletter
\newtheorem*{rep@theorem}{\rep@title}
\newcommand{\newreptheorem}[2]{%
\newenvironment{rep#1}[1]{%
 \def\rep@title{#2 \ref{##1}}%
 \begin{rep@theorem}}%
 {\end{rep@theorem}}}
\makeatother

\definecolor{RedOrange}{cmyk}{ 0, 0.77, 0.87, 0}
\definecolor{RoyalPurple}{cmyk}{ 0.84, 0.53, 0, 0}
\definecolor{YellowGreen}{cmyk}{ 0.44, 0, 0.74, 0}
\definecolor{Fuchsia}{cmyk}{ 0.47, 0.91, 0, 0.08}
\definecolor{Blue}{cmyk}{ 0.84, 0.53, 0, 0}
\definecolor{BlueViolet}{cmyk}{ 0.84, 0.53, 0, 0}
\definecolor{Black}{cmyk}{ 0.75, 0.68, 0.67, 0.9}
\usepackage{xcolor}
\usepackage{verbatim}
\usepackage{amsmath}
\usepackage{amssymb, mathrsfs}
\usepackage{amsbsy}

\usepackage{amscd}
\usepackage{amsthm}
\usepackage[english]{babel}
\usepackage{todonotes}
\usepackage[T1]{fontenc}
\usepackage{caption}
\usepackage{subcaption}
\usepackage{color}

\newcommand{\R}{\mathbb{R}}

\renewcommand{\O}{\mathbb{O}}

\newcommand{\N}{\mathbb{N}}

\newcommand{\E}{\mathbb{E}}
\newcommand{\Z}{\mathbb{Z}}
\newcommand{\be}{\mathbf{e}}

\renewcommand{\P}{\mathbb{P}}
\newcommand{\pr}{\mathbb{P}}

\newcommand{\kA}{\mathcal{A}}
\newcommand{\kB}{\mathcal{B}}
\newcommand{\kC}{\mathcal{C}}

\newcommand{\sC}{\mathscr{C}}
\newcommand{\sO}{\mathbb{G}}

\newcommand{\kR}{\mathcal{R}}
\newcommand{\kO}{\mathcal{O}}
\newcommand{\kP}{\mathcal{P}}

\newcommand{\kE}{\mathcal{E}}
\newcommand{\cE}{\mathcal{E}}

\newcommand{\kV}{\mathcal{V}}
\newcommand{\kW}{\mathcal{W}}

\newcommand{\kU}{\mathcal{U}}

\newcommand{\rmd}{{\mathrm{d}}}
\newcommand{\bef}{{\mathbf{e}}}

\newcommand{\cl}{\textrm{cl}}

\newcommand{\rmA}{{\mathrm{A}}}
\newcommand{\rmT}{{\mathrm{T}}}
\newcommand{\rmTp}{{\mathrm{\tilde{T}}}}

\newcommand{\rmTMK}{{\mathrm{T}_M^{\Lambda_K}}}

\newcommand{\rmD}{{\mathrm{D}}}
\newcommand{\aAn}{{ \mathrm{A}_N(e)}}

\newcommand{\aA}{{ \mathrm{A}}}
\newcommand{\eb}{{\mathbf{e}}}

\newcommand\1{{\mathbf 1}}
\newcommand\p{{\mathbf p}}

\newcommand{\lin}{\left[\kern-0.15em\left[}
\newcommand{\rin} {\right]\kern-0.15em\right]}
\newcommand{\linf}{[\kern-0.15em [}
\newcommand{\rinf} {]\kern-0.15em ]}
\newcommand{\ilin}{\left]\kern-0.15em\left]}
\newcommand{\irin} {\right[\kern-0.15em\right[}

\def\be#1{\begin{equation*}#1\end{equation*}}
\def\ben#1{\begin{equation}#1\end{equation}}
\def\bea#1{\begin{eqnarray*}#1\end{eqnarray*}}

\def\al#1{\begin{align*}#1\end{align*}}
\def\aln#1{\begin{align}#1\end{align}}

\usepackage{constants}

\newconstantfamily{c}{symbol=c}
\newconstantfamily{a}{symbol=\alpha}
\newcommand{\secno}[1]{\thesection.\arabic{#1}}
\newconstantfamily{kE}{
symbol=\mathcal{E},
format=\secno,
reset={section}
}
\renewcommand{\hat}{\widehat}
\renewcommand{\tilde}{\widetilde}

\newtheorem{lem}{Lemma}[section]

\newtheorem{prop}[lem]{Proposition}
\newtheorem{thm}[lem]{Theorem}

\newtheorem{cor}[lem]{Corollary}

\newtheorem {defin}[lem] {Definition}
\newtheorem {rem}[lem] {Remark}

\newcounter{assu}
\usepackage{color}
\definecolor{lilas}{RGB}{182, 102, 210}

\numberwithin{equation}{section}

\def\be#1{\begin{equation*}#1\end{equation*}}
\def\ben#1{\begin{equation}#1\end{equation}}
\def\bea#1{\begin{eqnarray*}#1\end{eqnarray*}}

\def\ba#1{\begin{align*}#1\end{align*}}

\DeclareMathOperator{\Diam}{Diam}

\title{Lipschitz-continuity of time constant in generalized First-passage percolation}

\author[V.~H.~CAN]{Van Hao CAN}
\address[V.~H.~CAN]{Institute of Mathematics, Vietnam Academy of Science and Technology, 18 Hoang Quoc Viet, Cau Giay, Hanoi, Vietnam.}
\email{cvhao@math.ac.vn}

\author[S.~NAKAJIMA]{Shuta NAKAJIMA}
\address[S.~NAKAJIMA]{Graduate School of Science and Technology, Meiji University, Kanagawa 214-8571, Japan.}
\email{njima@meiji.ac.jp}

\author[V.~Q.~NGUYEN]{Van Quyet NGUYEN}
\address[V.~Q.~NGUYEN]{Institute of Mathematics, Vietnam Academy of Science and Technology, 18 Hoang Quoc Viet, Cau Giay, Hanoi, Vietnam.}
\email{nvquyet@math.ac.vn}

\begin{document}
\maketitle
\begin{abstract}
In this article, we consider a generalized First-passage percolation model, where each edge in $\Z^d$ is independently assigned an infinite weight with probability $1-p$, and a random  finite weight otherwise. The existence and positivity of the time constant  have been established in  \cite{cerf2016weak}.  Recently, using sophisticated multi-scale renormalizations, Cerf and Dembin \cite{cerf2022time} proved that the time constant of chemical distance in super-critical percolation is Lipschitz continuous. In this work, we propose a different approach leveraging lattice animal theory and a simple one-step renormalization with the aid of Russo's formula, to show the Lipschitz continuity of the time constant in generalized First-passage percolation. 
\end{abstract}
\section{Introduction}
\subsection{Model and main results} First-passage percolation (FPP), which was introduced by Hammersley and Welsh in the 1960s, serves as a prototype for models of random growth or infection models. Let \(d \geq 2\) and \( (\Z^d, \mathcal{E}(\Z^d)) \) represent the \(d\)-dimensional integer lattice, where the edge set $\kE(\Z^d)$ consists of pairs of nearest neighbors in $\Z^d$. To each edge $e \in \kE(\Z^d)$, we assign a random variable $\omega_e$ with values in $[0,\infty)$, assuming that the family $(\omega_e)_{e \in \kE(\Z^d)}$ is independent and identically distributed. The random variable $\omega_e$ can be interpreted as the time needed for the infection 
to cross the edge $e$. We define a random pseudo-metric $\rmT$: for any
pair of vertices $x, y \in \Z^d$, $\rmT(x,y)$ is the shortest time to go from $x$ to $y$. The main object of FPP is to know how the infection grows in the lattice, or equivalently how is the asymptotic behavior of the passage time $\rmT(0,x)$ as $||x||_{\infty}$ tends to infinity. There has been a great and consistent interest of mathematicians for more than sixty years to answer this question, see, for instance, \cite{auffinger201750} and references therein. While most studies focus on the case of finite edge weight, i.e., $\omega_e$ takes a value in $[0,\infty)$, recently there have been several results on the behavior of generalized models allowing the infinite value, see e.g., \cite{garet2004asymptotic,cerf2016weak}. The emergence of infinite weight can explain the situation that some edges in the lattice are not available for the spread of infection. 

In this paper, we consider a generalized FPP that is mixed from the Bernoulli percolation and classical FPP. More precisely, given $F$ a distribution supported on $[0,\infty)$, and $p\in [0,1]$, we define a new distribution $F_p$ by
$$F_p:=pF+(1-p)\delta_{\infty}.$$
 Let $\tau:=(\tau_e)_{e \in \cE(\Z^d)}$ be a family of edge-weights with the same distribution $F_p$, interpreted as the time to pass each edge in $\Z^d$.
The usual first passage time $\rmT(x,y)$ on $\Z^d$ for $x,y\in \Z^d$ is defined by 
$${\rm T}(x,y):=\inf_{\gamma:x\to y} {\rm T}(\gamma):=\inf_{\gamma:x\to y}\sum_{e\in\gamma}\tau_e,$$
where the infimum is taken over all paths from $x$ to $y$ in $\Z^d$. We impose the following constraint on $p$ and $F$:
\ben{
p > p_c(d) > F(0),
}
where $p_c(d)$ is the critical parameter of Bernoulli percolation on $\Z^d$. The condition $p>p_c(d)$ guarantees the unique infinite cluster composed of finite weight edges, while the assumption $F(0)<p_c(d)$ rules out the possibility of having an infinite cluster with zero weight. Since the passage time $\rmT(x,y)$ may take the infinite value (when $x$ and $y$ are not connected by a path of finite weight edges), we consider a modification as follows. Let $\kC_{p}$ denote the unique infinite cluster of edges with finite weights. Given points $x,y\in \R^d$, we define the regularized passage time as
\be{
\rmTp (x,y) := \rmT([x]_p,[y]_p),
}
where $[x]_p$ denotes the ${\rm d_1}$-closest point to $x$ in $\kC_p$ with a deterministic rule breaking ties. Traditionally, the main object of interest in generalized FPP is the asymptotic behavior of $\rmTp$. Particularly, the weak law of large numbers was obtained in \cite{grimmett1990supercritical,cerf2016weak}: there exists a constant $\mu_p \in [0,\infty)$ such that
\begin{align}\label{limit}
 \lim_{n \to \infty} \frac{\rmTp(0, n\mathbf{e}_1)}{n} = \mu_p \qquad \text{ in probability},
\end{align}
where $\mathbf{e}_1$ is the first unit vector in $\R^d$.   Moreover, Garet and Marchand \cite[Remark 1]{garet2004asymptotic} proved that if $\E[\tau_e^{2+\delta} \1_{\tau_e < \infty}] < \infty$ with some $\delta >0$, then the convergence in \eqref{limit} holds true almost surely and in $L_1$.  Our first result is the strong law of large numbers for the regularized first passage time assuming solely the finiteness of first moment of $\tau_e \1_{\tau_e < \infty}$. We prove it   in Appendix \ref{app:slln}. 
\footnote{Although the proof of Theorem \ref{Theorem: time constant} is based on classical Kingman's sub-additive ergodic theorem and is quite simple, we could not find any reference for it.}
\begin{thm}{\bf (SLLN)} 
\label{Theorem: time constant}
If $p > p_c(d)$ and $\E[\tau_e \1_{\tau_e < \infty}] < \infty $, then 
\begin{align*}
 \lim_{n \to \infty} \frac{\rmTp(0, n\mathbf{e}_1)}{n} = \mu_p \qquad \text{ a.s. and in } L_1.
\end{align*}
\end{thm}
\iffalse
Our second result shows the Lipschitz continuity of the function $p \mapsto \mu_p$. 
 The continuity and regularity of time constant of FPP and chemical distance in super-critical percolation have been investigated since 1980s. For instance, the continuity has been discussed in \cite{cox1980time,cox1981continuity,garet2017continuity}, while    the regularity has been addressed in \cite{dembin2021regularity,cerf2022time,kubota2022comparison}. As we will see in the outline of the proof, the time constant can be represented as the limit of a truncated passage time. Hence, the moment condition on weight is not required in Theorem \ref{Theorem: main}. We notice also that Cerf and Dembin \cite{cerf2022time} have established the Lipschitz continuity of the time constant of the chemical distance, i.e., $F=\delta_1$. 
 Going further, the authors also claim a quantative estimate of difference of time constants for general distributions $F$ and $G$ (that includes Theorem~\ref{Theorem: main} below), though they do not give detailed proof. 
Their proof relies on 
 a sophisticated multi-scale renormalization. In this paper, we propose a different approach leveraging a tesselation argument with a simple one-step renormalization. 
\begin{thm} {\bf (Lipschitz continuity)}\label{Theorem: main}
For all $p_0 > p_c(d)$, there exists $C=C(p_0)>0$ such that
$$|\mu_p-\mu_q|\leq C |p-q|,\quad\forall p,q\in [p_0,1].$$
\end{thm}
\fi
The continuity and regularity of the time constant of First-passage percolation and chemical distance in super-critical percolation have been subjects of investigation since the 1980s. The continuity has been explored in works \cite{cox1980time,cox1981continuity,garet2017continuity}, while the regularity has been addressed in \cite{dembin2021regularity,cerf2022time,kubota2022comparison}.\footnote{Note that in \cite{kubota2022comparison}, a distribution defined as $F_p=p\delta_0+(1-p)\delta_1$ was considered, and explicit bounds for the Lipschitz constants were obtained.} In particular, Cerf and Dembin \cite{cerf2022time} have established the Lipschitz continuity of the time constant for the chemical distance, i.e., \( F = \delta_1 \).  Going further, the authors also claim a quantitative estimate of difference of time constants for two distributions (that includes Theorem~\ref{Theorem: main} below), though they do not give detailed proof. 

 In this context, we present our main result as follows:
\begin{thm} {\bf (Lipschitz continuity)}\label{Theorem: main}
For all \( p_0 > p_c(d) \), there exists a constant  \( C = C(d,p_0,F) > 0 \) such that for all \( p, q \) in the interval \( [p_0,1] \),
$$|\mu_p - \mu_q| \leq C |p - q|.$$
\end{thm}
%This theorem demonstrates the Lipschitz continuity of the function mapping \( p \) to \( \mu_p \).
Notably, the time constant can be expressed as the limit of a truncated passage time defined below, which implies that the moment condition on weight is not necessary for this theorem.
\subsection{Outline of the proof}  The proofs in \cite{cerf2022time} utilizes a sophisticated multi-scale renormalization technique. However, in our paper, we propose an alternative approach that employs lattice animal theory combined with a straightforward one-step renormalization process. Let us explain the outline of the proof here.

Let $M:=M_n:=(\log n)^3$ and $K :=K_n:= n^2$.  We denote by $\rmTMK(x,y)$ the first passage time between $x$ and $y$ associated with the truncated weights $(\tau^M_e)_{e \in \kE(\Z^d)}$ using only paths inside $\Lambda_K$, where $\tau^M_e: =\tau_e \wedge M$. Then the proof of Theorem \ref{Theorem: main} is decomposed into two steps: \\

\noindent \textbf{Step 1 (Time constant as the limit of truncated passage time)}: We aim to show
\begin{align}\label{Time constant cut off at M}
 \lim_{n \to \infty} \dfrac{\E \left[\rmT_M^{\Lambda_K}(0,n\mathbf{e}_1)\right]}{n} = \mu_{p}.
\end{align}
The proof goes as follows. Let $\lambda$ be a large positive constant and $q:=\pr(\tau_e \leq \lambda)$, see Appendix \ref{app:lamd} for the choice of $\lambda$. We consider the percolation of $q$-open edges consisting of $\{e \in \kE(\Z^d):~\tau_e \leq \lambda\}$ and use similar notations, such as $\kC_q$ and $[x]_q$, for this percolation. Note that $\kC_q \subset \kC_p$, and the vertices in $\kC_q$ can be connected to each other along paths whose weights are at most $\lambda$. According to \cite[Lemma 2.11]{garet2017continuity}, we have
\be{
 \lim_{n \to \infty} \frac{\rmT([0]_q,[n\mathbf{e}_1]_q)}{n} =\mu_p \quad \text{ a.s. and in } L_1.
}
In this step, we further aim to show 
\be{
 \E \left[\left|\rmT([0]_q, [n\mathbf{e}_1]_q)-\rmT_M^{\Lambda_K}(0,n\mathbf{e}_1)\right| \right] = \kO( \lambda M).
}
To prove this estimate, we introduce the notation of \textbf{effective radius} $(R_e)_{e\in \kE(\Z^d)}$ in  Section \ref{sec:resampling}. Roughly speaking, given an edge $e$ belonging to a geodesic of  the truncated passage times, $R_e$ measures the length of $q$-open path for bypassing $e$. Under the event that $\{R_e \leq (\log n)^{5/2}, \, \forall e \in [-n^2,n^2]^d\}$ which occurs with overwhelming probability, we  show that $\big|\rmT([0]_q, [n\mathbf{e}_1]_q)-\rmT_M^{\Lambda_K}(0,n\mathbf{e}_1)\big| =\kO(\lambda M)$. In particular, we have \eqref{Time constant cut off at M}. We refer to Section \ref{sec:comp} for the details.\\

\noindent \textbf{Step 2 (Linear bound via Russo's formula)}: Let $\rmT_{M,\pm,e}^{\Lambda_K}(0,n\mathbf{e}_1)$ be the first passage time when the weight of the edge \( e \) is set to \( M \) for \( + \) and \( 0 \) for \( - \), respectively. We take $\gamma$ to be a geodesic of ${\rmT_M^{\Lambda_K}(0,n\mathbf{e}_1)}$, and we define 
$\Delta_e \rmT_M^{\Lambda_K}(0,n\mathbf{e}_1) :=\rmT_{M,+,e}^{\Lambda_K}(0,n\mathbf{e}_1)-\rmT_{M,-,e}^{\Lambda_K}(0,n\mathbf{e}_1)$. We aim to show 
\ben{ \label{russo}
\left| \frac{{\rm d} \E \left[{\rmT_M^{\Lambda_K}(0,n\mathbf{e}_1)} \right]}{{\rm d}p} \right|\leq \E \left[ \sum_{e \in \gamma} \Delta_e {\rmT_M^{\Lambda_K}(0,n\mathbf{e}_1)} \right] \leq \kO(1) \E \left[ \sum_{e \in \gamma} R_e \right] \leq \kO(n).
} The first inequality follows from a standard application of Russo's formula. The second inequality simply follows from the construction of effective radius appearing above. The proof of the last inequality in \eqref{russo} uses properties of effective radius, i.e., a local dependence and a good probability decay, and lattice animal theory.\\ %More precisely,  thanks to the properties (i) and (ii) below, using lattice animal theory, we can get the linear bound for the sum of effective radii along $\gamma$:
\iffalse

(i) The radii $(R_e)_{e \in \kE(\Z^d)}$ are locally dependent in the sense that for all $e \in \kE(\Z^d)$ and $t \geq 1$ the event $\{R_e \leq t\}$ depends solely on the edge weights $\{\tau_{e'}: ||e-e'||_{\infty} \leq Ct \}$ with some constant $C$;

(ii) $R_e$ has a good exponential tail: $\pr(\tau_e \geq t) \leq \exp(-c\sqrt{t})$ for all $t \in [cM^2]$, where $c$ is a positive constant.\\
\fi
%Finally, combining Step 1 and Step 2 yields the Lipschitz continuity of the time constant. \\

The effective radius along with the utilization of lattice animal theory is proved to be robust in estimating the effect of flipping edge in percolation. In fact, we can use these ingredients to establish the sub-diffusive concentration of chemical distance in Bernoulli supercritical percolation as in \cite{CNg}. Nevertheless, the effective radius in \cite{CNg} is defined in a slightly more complex manner to enable the construction of a bypass for an edge along a mixed path composed of several open paths and geodesics. Furthermore, it also satisfies a more robust large deviation estimate compared to Proposition \ref{Lem: effective radius}.

\iffalse
Given two distributions $F,G$ supported on $[0,\infty)$, we define
$$\kR^{F\otimes G} _\delta:=\{(p,q)\in [0,1]^2:~pF(0)+qG(0)<p_c-\delta,\quad p+q>p_c+\delta\}.$$
\begin{cor} 
\footnote{To prove the corollary in this regime, I want to extend Thm 1.2 as follows: Let $F$ a measure on $[0,\infty)$ and $\lim_{x\to\infty}F(\R)=:F(\infty)\leq 1$. Let us define the distribution $F_p$ supported on $[0,\infty]$ satisfying
$$F_p:=p\delta_{0}+p F+(1-p-p_0)\delta_\infty,\qquad (x\in\R).$$
 Then, for all $\delta>0$, there exists $C>0$ such that for all $p,q\in \kR^F_\delta:=\{p\in [0,1]:~p<p_c-\delta,\quad p+F(\infty)>p_c+\delta\}$,
$$|\mu_p-\mu_q|\leq C|p-q|.$$}

 We define $F_{p,q}:=pF^1(x)+qF^2(x)+(1-p-q)\delta_\infty$ and $\mu_{p,q}$ be the corresponding time constant. Then, for all $\delta>0$ there exists $C>0$ such that for all $(p,q),(p',q')\in \kR^{\otimes 2} _\delta$,
 $$|\mu_{p,q}-\mu_{p',q'}|\leq C(|p-p'|+|q-q'|).$$
\end{cor}
\begin{proof}
Note that $F_{p,q}$ is stochastically dominated by $F_{(1-\epsilon)p,(1-\epsilon)q}$. Let $\epsilon>0$ be such that $p'\geq (1-\epsilon)p$ and $q'\geq (1-\epsilon)q$.\footnote{We need to care in which case we can take such $\epsilon$} By the above theorem with $(1-\epsilon)(p+q)$ and $(pF+qG)/(p+q)$ in place of $p'$ and $F$, we have
 $$\mu_{p',q'}-\mu_{p,q}\leq \mu_{(1-\epsilon)p,(1-\epsilon)q}-\mu_{p,q} \leq C(|p-p'|+|q-q'|).$$
 The converse inequality holds by the same argument.
\end{proof}
\fi
\subsection{Notation} We summarize some notation frequently used throughout the paper.

$\bullet$ \textit{Integer interval}. We define $[a]:=[1,a]\cap \Z$ for all $a\geq 1$.

$\bullet$ \textit{Box and its boundary}. For every $x \in \Z^d$ and $t >0$, we define $\Lambda_{t}(x): = x+  [-t,t]^d \cap \Z^d $ the box with center $x$ and radius $t$. For simplicity, we write $\Lambda_t := \Lambda_t(0)$. We define the boundary of $\Lambda_t(x)$ as $\partial \Lambda_t(x):=\Lambda_t(x) \setminus \Lambda_{t-1}(x)$.

$\bullet$ \textit{Edge set}. Given a set $A \subset \Z^d$, we denote by $\kE(A)$ the set of edges both of whose endpoints belong to $A$.

$\bullet$ \textit{Set distance}. For $X,Y \subset \Z^d$, we consider several kinds of distance between $X$ and $Y$ as
 \begin{align*}
 {\rm d}_\star (X,Y) := \min \{\|x-y\|_\star:x \in X,y \in Y \}, \qquad \star \in \{1, 2, \infty\}.
 \end{align*}
 
$\bullet$ \textit{Path} and \textit{open path}. We say that a sequence $\gamma = (v_0,\ldots, v_n)$ is a {\bf path} if $ |v_i - v_{i-1}|_1 =1$ and $v_i\neq v_j$ for all $i\neq j \in [n]$. 
%{\QN For any path $\gamma = (v_0,\ldots, v_n)$, we denote by $\gamma_{v_i,v_j} := (v_i,v_{i+1},\ldots,v_j)$ the sub-path of $\gamma$ from $v_i$ to $v_j$ for some $ 0\leq i < j \leq n$}. 
Given $A\subset \Z^d$, let $\kP(A)$ denote the set of all paths inside $A$. Given a Bernoulli percolation on $\Z^d$ with parameter $p$, we say that a path is {\bf $p$-open} if all of its edges are open. An {\bf open cluster} is a maximal connected component in the percolation. An open cluster $\kC$ is called a $q$-crossing in $\Lambda$ if in each direction there is an open path in $\kC$ connecting the two opposite faces of $\Lambda$. In that case, we write $q$-crossing cluster $\kC\subset \Lambda$.  

 $\bullet$ \textit{Geodesic} and \textit{truncated passage time} : Let $\rmT$ be the first passage time associated with weights $(\omega_e)_{e \in \kE(\Z^d)}$. Given $x,y \in \Z^d$, a path $\gamma$ between $x$ and $y$ is termed a \textbf{geodesic} of $\rmT$ if its passage time matches $ \rmT(x,y)$, i.e., $\rmT(\gamma):=\sum_{e \in \gamma} \omega_e =\rmT(x,y)$. Given $H>0$ and $A\subset \Z^d$, we define the \textbf{truncated passage time}, denoted by $\rmT^A_H$, as the first passage time associated with the truncated weights $(\omega_e \wedge H)_{e \in \kE(Z^d)}$ using only paths inside $A$. %That means for all $x,y \in A$, $\rmT_H^A(x,y) = \inf \{\sum_{e \in \gamma} (\omega_e \wedge H) : \gamma \textrm{ is a path inside $A$ from $x$ to $y$} \}$.
  When $A=\Z^d$, we write $\rmT_H:=\rmT_H^{\Z^d}$.
\subsection{Organization}
The paper is organized as follows. In Section \ref{sec:ing}, we introduce the main ingredients of proof including Russo's formula, effective radius, and lattice animal theory. In Section \ref{Section: Lipschitz continuity}, we prove Step 1 and Step 2 using the elements prepared in Section 2. In the Appendix, we prove the strong law of large numbers of the passage time (Theorem \ref{Theorem: time constant}), Russo's formula and properties of effective radius. 
\section{Main ingredients of proof} \label{sec:ing}
In this section, we introduce three main elements in proving the Lipschitz continuity. The first result is Russo's type formula (Lemma~\ref{lemrusso}). The second result considers the effects of resampling an edge (Propositions~\ref{Lem: effective radius} and \ref{prop:bypass}), and the third result provides an upper bound on the total cost of resampling along a random path using the lattice animal theory (Corollary~\ref{corre}).  Although they have been already investigated in previous research, we provide the proofs of these results in Appendix for the completeness of the paper. 
\subsection{Russo's type formula}
Let $L \in \R_{+} \cup \{\infty\}$. Let $\nu$ be a random variable with the distribution $G$ supported in $[0,L]$. For $p \in (p_c(d) ,1)$, we define the distribution $G_p$ on $[0,L]$ by
\begin{align*}
 G_p := p G + (1-p) \delta_{L},
\end{align*}
where $\delta_L$ stands for the Dirac delta distribution at $L$.
\begin{lem}
\label{lemrusso} Let $E$ be a finite set, $\xi=(\xi_e)_{e \in E}$ i.i.d. random variables with the common distribution $G_p$, and $X : [0, L]^E \to \R$  be a function such that $X(\xi)$ is integrable. Suppose that  $\xi^{+,e}$ and $\xi^{e}$ are obtained from $\xi$ by replacing $\xi_e$ with $L$ and with $\nu$ respectively, where $\nu$ is an independent random variable with distribution $G$. Then, we have
\begin{align*}
 \dfrac{{\rm d} \E[X(\xi)]}{{\rm d} p} = \sum_{e \in E} (\E[X(\xi^e)]- \E[X(\xi^{+,e})]).
\end{align*}
\end{lem}
\subsection{The effect of resampling edges}\label{sec:resampling}
As we will see in the next section, using Russo's type formula (Lemma~\ref{lemrusso}), the problem of Lipschitz continuity of time constant can be reduced to controlling the effect of resampling the edges along the geodesics. Given an edge $e$, we introduce the {\bf effective radius} $R_e$, which measures the change of chemical distance when flipping the state of $e$ from open to closed. 

Given a coupling of Bernoulli percolation models for parameters $p$, a path $\gamma$ is called $p$-open if all of its edges are open in the corresponding percolation with parameter $p$. We define the set of $p$-open paths in $A\subset \Z^d$ by 
 \be{
 \mathbb{O}_p{(A)}:=\{\gamma\in \kP(A):\textrm{ $\gamma$ is $p$-open}\}.
 }
 For $A,B,U \subset \Z^d$, we define the {\bf chemical distance}
 \begin{align*}
 \rmD^U_{p}(A,B):= \inf \{ |\gamma|: x \in A,\, y\in B,\,\gamma \text{ is a $p$-open path from } x \text{ to } y\text{ inside }U\}.
 \end{align*}
%For $U \subset \Z^d$, let $\rmD_{p}^{U}$ be the graph distance using only $p$-open paths inside $U$.
When $U=\Z^d$, we simply write $\rmD_p$ for $\rmD_p^{\Z^d}$.{ Given $p\in [0,1]$ and $\lambda \in \R$, we define
\ben{
q:= q(p, \lambda):= \pr(\tau_e \leq \lambda)=p F([0, \lambda]).
}
Let $\delta_0$ be a sufficiently small positive constant as in Lemma \ref{apqn} below. Given $p_0 \in (p_c(d),1]$, we define $q_0:= \frac{p_0+p_c(d)}{2}$ and take   $\lambda=\lambda(p_0,F)$ sufficiently large  such that $F([0,\lambda]) \geq \max\left\{ \frac{q_0}{p_0},1-\delta_0\right \}$, which implies
%\footnote{Note that if $F$ has a bounded support, then we simply take $q=p$.} 
 \ben{ \label{dolam}
 q_0 \leq q \leq p \leq q + \delta_0 \, \, \forall \,  p\in [p_0,1].
}
We say that an edge $e$ is $q$-open or $p$-open if $\tau_e \leq \lambda$ or $\tau_e < \infty$, respectively. We call $q$-percolation and $p$-percolation the associated percolation models. Let $\kC_q$ and $\kC_p$ be the corresponding infinite clusters. We will see in  Appendix \ref{app:lamd} that the condition \eqref{dolam} assures that a large cluster in $\kC_q$ and a  long path in $\kC_p $ would intersect with high probability.} %From now on, we fix such $\lambda$ and $q$.}
  Given an edge $e\in \kE(\Z^d)$, we fix a rule to write $e=(x_e,y_e)$ so that $\|x_e\|_1<\|y_e\|_1$. For $N \geq 1$, and $e=(x_e,y_e)$, we define $\Lambda_N(e):=\Lambda_N(x_e)$, and an annulus
\begin{align}
\aAn
 := \Lambda_{3N}(e) \setminus \Lambda_N(e).
\end{align} 
We say that $\gamma$ is a crossing path of $ \aAn$ if $\gamma$ is a path 
 inside $ \aAn$ that joins $\partial \Lambda_{N}(e)$ and $\partial \Lambda_{3N}(e)$. Let $\sC(\aAn)$ be the collection of all crossing paths of $\aAn$.  Given $H>0$ and $A\subset \Z^d$, recall that $\rmT^A_H$ is the first passage time associated with the truncated weights $(\tau_e \wedge H)_{e \in \kE(Z^d)}$ using only paths inside $A$. For $u,v \in A$, we define the set of geodesics of $\rmT_H^A(u,v)$ as
 \be{
 \sO_H(u,v;A) :=\{\gamma =(u,\ldots,v) \in \kP(A): \rmT_H(\gamma)=\rmT_H^{A}(u,v) \}.
 }
We also define 
$$\sO_H(A):=\bigcup_{u,v \in A} \sO_H(u,v;A).$$
If $A=\Z^d$, we simply write $\sO_H(x,y)$ for $\sO_H(x,y;\Z^d)$ and write $\sO_H$ for $\sO_H(\Z^d)$.
 \begin{rem} \label{rem:omkx} 
 Given $B \subset A \subset \Z^d$ and $H>0$, if $\gamma \in \sO_H(A)$ and $\pi$ is a sub-path of $\gamma$ such that $\pi \subset B$, then $\pi \in \sO_H(B)$. We note that 
  $\sO_H(A)$ is measurable with respect to the weights of edges inside $A$. 
 \end{rem}
  \begin{figure}[h]
  \centering
  \includegraphics[width=0.8\linewidth]{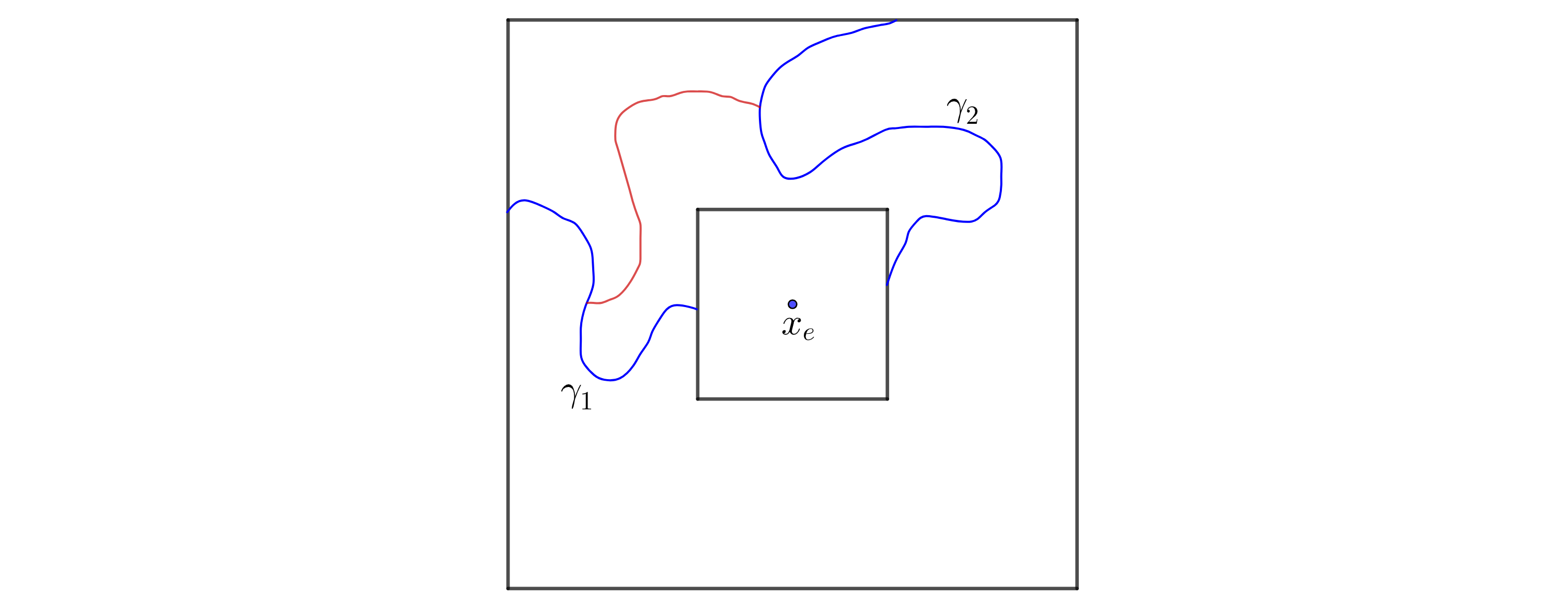}
  \caption{Illustration of the event inside the definition of $R_e$: the blue lines represent the crossing paths $\gamma_1$ and $\gamma_2$ of $\aAn$; the red line represents a geodesic of $\rmD_q^{\aAn}(\gamma_1,\gamma_2)$.}
  \label{fig:eradisu}
\end{figure}
  Let $C_*\geq 3$ be a constant. For each $e \in \kE(\Z^d)$, we define the $q$-\textbf{effective radius} of $e$ (see Figure \ref{fig:eradisu} for illustration) as
 %\footnote{Note that if $F$ has a bounded support and $H$ is large enough, then a geodesic of $\rmT_H^A(u,v)$ is a $p$-open path and $D_q=D_p$. Hence, in this case, the estimate  below simply follows from large deviation estimates, e.g., \cite[(4.49)]{antal1996chemical}.}
\begin{align*}
 R_e := R_e(C_*,H):=\inf \left\{N \geq 3: \forall \gamma_1, \gamma_2 \in \sO_{H}(\Lambda_{C_*N}(e)) \cap \sC(\aAn),\, 
 {\rm D}_{q}^{\aAn}(\gamma_1,\gamma_2) \leq C_*N \right\}.
\end{align*}
\begin{rem} \label{rem:weak}
By the definition of effective radius and Remark \ref{rem:omkx}, for all $e \in \kE(\Z^d)$ and $t \geq 1$ the event $\{R_e=t\}$ depends solely on the states of edges within the box $\Lambda_{C_*t}(e)$. 
\end{rem}
 We fix $\lambda=\lambda(p_0,F), q=q(p,\lambda)$ as in \eqref{dolam} throughout the paper. The following proposition gives a large deviation estimate for effective radii.
 \begin{prop}\label{Lem: effective radius}
 \footnote{A stronger (exponential) bound for Proposition~\ref{Lem: effective radius} is obtained in \cite[Section 2]{CNg}, though the present estimate is sufficient for our current purpose.}
Let $p_0 \in (p_c(d),1]$. There exist $C_* \geq 3$ and $ c \in (0,1)$ depending on $d, p_0$ such that for all $p \in [p_0,1]$ and $H >0$,  
 \begin{align*}
 \pr(R_e \geq t) \leq c^{-1} \exp(-c \sqrt{t}) \qquad \forall e \in \kE(\Z^d), \quad \forall \, t \in [c H^2].
 \end{align*}
\end{prop}
 We fix $C_*$ as in Proposition \ref{Lem: effective radius}. Given an edge in a geodesic, the following proposition allows us to build upon this geodesic a bypass to avoid this edge.
%
%\begin{rem} 
%\end{rem}
\begin{prop} \label{prop:bypass}
 Let $x,y \in \Z^d$ and $\gamma\in \sO_{H}(x,y)$ be a geodesic of $\rmT_H(x,y)$. Suppose that $e \in \gamma$ is an edge satisfying $x,y \not \in \Lambda_{3R_e}(e)$. Then there exists another path $\eta_e$ between $x$ and $y$ such that: 
 \begin{itemize}
 \item[(a)] $\eta_e \cap \Lambda_{R_e-1}(e)=\emptyset$ and $\eta_e \setminus \gamma$ consists only of $q$-open edges; 
 \item[(b)] $|\eta_e \setminus \gamma| \leq C_* R_e$.
 \end{itemize} 
\end{prop}
The proofs are postponed until Appendix since they are standard in percolation theory.
\subsection{Lattice animals of dependent weight} 
To manage the cumulative cost of edge resampling, we aim to estimate the sum of effective radii along a random path. While these effective radii are not mutually independent, their interdependence is relatively local (Remark \ref{rem:weak}). We utilize lattice animal theory to provide an upper bound for the sum of these radii. We first revisit a result that controls the total weight of paths in dependent environments in \cite{CNg} using the theory of greedy lattice animals. 

Let $\mathcal{P}_L$ denote the set of all paths $\gamma$ starting at $0$ and satisfying $|\gamma|\leq L$. Given $A >0$ and  $N\in \N$,   a collection of Bernoulli random variables $(I_{e,N})_{ e \in \cE(\Z^d)}$ is called  $AN$-dependent if for all $e \in \cE(\Z^d)$, the variable $I_{e,N}$ is independent of  the random variables $(I_{e',N})_{e' \not \in \kE(\Lambda_{AN}(e))}$.
For any $\gamma\in \kP_L$, we define
\begin{align*}
 {\rm \Gamma}(\gamma) := \sum_{e\in \gamma}I_{e,N}, \quad {\rm \Gamma}_{L,N} := \max_{\gamma \in \mathcal{P}_L} {\rm \Gamma}(\gamma).
\end{align*}
\begin{lem} \cite[Lemma 2.6]{nakajima2019first} \label{lemmaxbound}
 Let $A>0$, $N\in \N$ and   a collection of $AN$-dependent Bernoulli random variables $(I_{e,N})_{ e \in \cE(\Z^d)}$. There exists a positive constant $C$ depending on $A, d$ such that for all $L \in \N$,
\be{
\E[{\rm \Gamma}_{L,N}] \leq C L N^{d} q_N^{1/d}, \quad \textrm{where} \quad q_N :=\sup_{e \in \kE(\Z^d)} \E[I_{e,N}].
}
\end{lem}

\begin{proof}
 We give a simplified proof here. Given $A\in \N$, let us consider a decomposition $\kE(\Z^d)=\bigcup_{i=1}^{(2d A N)^d} E_i$ such that for each $E_i$,  $\rmd_{\infty}(\{x,y\},\{x',y'\})\geq 2AN$ for all $e=( x,y)\neq e'=( x',y') \in E_i$ (see 
 \cite[Lemma 2.6]{nakajima2019first} for a concrete example). This implies that for all $1 \leq i \leq (2dAN)^d$, the random variables $(I_{e,N})_{ e \in E_i}$ are independent. Let $(\bar{I}_{e,N})_{e\in \kE(\Z^d)}$ be i.i.d. Bernoulli random variables with mean $q_N$. Notice that the family of random variables $(I_{e,N})_{ e \in E_i}$ is stochastically dominated by $(\bar{I}_{e,N})_{ e \in E_i}$, since $ \E[I_{e,N}] \leq q_N$  for all $e \in \kE(\Z^d)$. Therefore, for any $L\in \N$, we have 
 \al{
\E[{\rm \Gamma}_{L,N}] &\leq \sum_{i=1}^{(2d A N)^d} \E\left[\max_{\gamma\in \kP_L} \sum_{e\in \gamma\cap E_i} I_{e,N}\right]  \leq \sum_{i=1}^{(2d A N)^d} \E\left[\max_{\gamma\in \kP_L} \sum_{e\in \gamma\cap E_i} \bar{I}_{e,N}\right] \leq \sum_{i=1}^{(2d A N)^d} \E\left[\max_{\gamma\in \kP_L} \sum_{e\in \gamma} \bar{I}_{e,N}\right].
 }
 Using \cite[Lemma 6.8]{damron2015sublinear}, 
 $\E\left[\max_{\gamma\in \kP_L} \sum_{e\in \gamma} \bar{I}_{e,N}\right] \leq \kO( L q_N^{1/d}),$ which yields the claim.
\end{proof}
 The following result controls the total weight of an arbitrary random path. %when the environment is finite-range dependent and has a good tail decay.
\begin{lem} \label{Lem: lattice animal} 
Let $A>0$ and $(X_e)_{e \in \cE(\Z^d)}$ be a family of non-negative random variables such that for all $e \in \cE(\Z^d)$ and $ N \in \N$, 
\begin{align}\label{Pd}
\textrm{the event } \{N-1 \leq X_e < N\} \textrm{ is independent of } (X_{e'})_{e'\in \kE(\Z^d \setminus \Lambda_{A N}(e))}.
 \end{align} 
We define \(q_N := \sup_{e\in \kE(\Z^d) } \pr(N-1 \leq X_e < N).\) Let $f:[0,\infty) \to [0,\infty)$ be a function satisfying
\begin{align}\label{hd}
 B:=\sum_{N=1}^{\infty} (f_*(N))^2 N^{d} q_N^{1/d}< \infty, \quad \textrm{ where} \quad f_*(N):=\sup_{N-1\leq x< N} f(x).
 \end{align} 
Then there exists $C=C(A,B)>0$ such that for all random paths $\gamma$ starting from $0$ in the same probability space as $(X_e)_{e \in \kE(\Z^d) }$, and $L \in \N$,
\begin{align*}
 \E\left[ \sum_{e \in \gamma} f(X_{e}) \right] \leq CL + C \sum_{\ell \geq L} \ell (\pr(|\gamma|=\ell))^{1/2}.
 \end{align*}
\end{lem}
\begin{proof}
 By Cauchy-Schwarz inequality, we have
 \begin{align}\label{separate}
 \E\left[ \sum_{e \in \gamma} f(X_{e}) \right] & =\E\left[ \sum_{e \in \gamma} f(X_{e})  \1_{(|\gamma| < L)} \right] + \E\left[ \sum_{e \in \gamma} f(X_{e}) \1_{|\gamma| \geq L} \right] \notag \\
 &\leq \E\left[ \max_{\gamma \in \mathcal{P}_L} \sum_{e \in \gamma} f(X_{e}) \right] + \sum_{\ell=L}^{\infty} \E\left[ \sum_{e \in \gamma} f(X_{e}) \1_{|\gamma| =\ell} \right] \notag \\
 & \leq \left(\E\left[ \left( \max_{ \gamma \in \mathcal{P}_{L}} \sum_{e \in \gamma} f(X_{e}) \right)^2 \right] \right)^{1/2}+ \sum_{\ell=L}^{\infty} \left(\E\left[ \Big( \max_{ \gamma \in \mathcal{P}_{\ell}} \sum_{e \in \gamma} f(X_{e}) \Big)^2 \right] \right)^{1/2} (\pr[|\gamma| =\ell])^{1/2}.
 \end{align}
Let $m\geq L$. By Cauchy-Schwarz inequality we know that $(\sum_{i=1}^n x_i)^2 \leq n \sum_{i=1}^n x_i^2$,
\begin{align} \label{max}
\E\left[ \left( \max_{ \gamma \in \mathcal{P}_{m}} \sum_{e \in \gamma} f(X_{e}) \right)^2 \right]= \E\left[  \max_{ \gamma \in \mathcal{P}_{m}} \left( \sum_{e \in \gamma} f(X_{e}) \right)^2 \right] \leq \E\left[ \max_{ \gamma \in \mathcal{P}_{m}} | \gamma| \sum_{e \in \gamma} f^2(X_{e}) \right] \leq m\E\left[ \max_{ \gamma \in \mathcal{P}_{m}} \sum_{e \in \gamma} f^2(X_{e}) \right].
\end{align}
Let $ I_{e,N} := \1_{N-1 \leq X_e < N}$. We have
$$\sum_{e \in \gamma} f^2(X_{e}) = \sum_{e \in \gamma} \sum_{N\geq 1}f^2(X_e)I_{e,N} \leq \sum_{N\geq 1}  (f_*(N))^2 \sum_{e \in \gamma} I_{e,N}.$$ Let $ \Gamma_{m,N}:= \max_{\gamma \in \mathcal{P}_m}\sum_{e \in \gamma} I_{e,N}.$ Therefore, 
\begin{align} \label{boudnmaxsquare1}
 \E \left[ \max_{\gamma \in \mathcal{P}_m} \sum_{e \in \gamma} f^2(X_{e}) \right] & \leq \E \left[\sum_{N\geq 1}  (f_*(N))^2 \max_{\gamma \in \mathcal{P}_m}\sum_{e \in \gamma} I_{e,N} \right] = \sum_{N\geq 1}  (f_*(N))^2 \E \left[ \Gamma_{m,N} \right].
\end{align}
 %Therefore, the condition \eqref{Ed} is satisfied for all $N \geq 1$.
 By Lemma \ref{lemmaxbound} with \eqref{Pd}, for all $N \geq 1$, $\E [\Gamma_{m,N}] =\kO(m) N^{d} q_N^{1/d}.$ Combined with \eqref{boudnmaxsquare1}, this yields
\begin{align*}
 \E \left[ \max_{ \gamma \in \mathcal{P}_{m}} \sum_{e \in \gamma} f^2(X_{e}) \right] & =\kO( m) \sum_{N \geq 1}  (f_*(N))^2 N^{d} q_N^{1/d} = \kO(m),
\end{align*}
by the assumption \eqref{hd} on $f$. Finally, combining this with \eqref{separate} and \eqref{max}, we derive the claim. \end{proof}
Applying Lemma \ref{Lem: lattice animal} with $X_e= R_e\1_{R_e \leq M}$, $A=2C_*$, and $f(x)=x$, since the conditions \eqref{Pd} and \eqref{hd} follow from Remark \ref{rem:weak} and Proposition \ref{Lem: effective radius} respectively, we have the following:
\begin{cor} \label{corre}
For any $C>0$, there exists $C'$ such that the following holds. For all $L\in\N$ and  a random path $\gamma$ starting from $0$ satisfying $\pr( |\gamma| =\ell) \leq \ell^{-5}$ for all $\ell\geq C L$, we have
 \be{
 \E\left[ \sum_{e \in \gamma} R_{e} \1_{R_e \leq M} \right] 
 \leq C' L.
 }
\end{cor}
\section{Lipschitz continuity of the time constant: Proof of Theorem \ref{Theorem: main} } \label{Section: Lipschitz continuity}
 In this section, we shall apply the results of effective radius to the truncated passage time $\rmT_M^{\Lambda_K}$. Recall that $q_0=\tfrac{p_0+p_c(d)}{2}>p_c(d)$ and $\lambda$ from Section 2, such that $F([0,\lambda])\geq \max
 \left \{\tfrac{q_0}{p_0},1-\delta_0\right\}$. This implies
  $$\forall p\in [p_0,1],\quad q=pF([0, \lambda]) \geq q_0, \quad
 q_0 \leq q \leq p \leq q + \delta_0,
 $$ 
 where $\delta_0$ is a positive constant as defined 
 in Lemma  \ref{apqn}.
\subsection{Length of geodesics}
We recall some estimates on the sizes of holes and chemical distances. 
\begin{lem} \cite[Theorem 2]{pisztora1996surface} \label{lem:hole}
There exists $c=c(q_0) \in (0,1)$ such that for all $t \geq 1$, 
 \begin{align} \label{hole}
 \pr \left( \Lambda_{t} \cap \kC_{q} = \emptyset \right) \leq \pr \left( \Lambda_{t} \cap \kC_{q_0} = \emptyset \right) \leq c^{-1}\exp(-c t^{d-1}).
 \end{align}
 Consequently, for all $x \in \Z^d$ and $t >0$,
 \begin{align}\label{Claim: 2e0}
 \pr ( \|x-[x]_q\|_\infty \geq t) \leq c^{-1} \exp (-c t^{d-1}). 
 \end{align}
 \end{lem}
\begin{lem} \cite[(4.49)]{antal1996chemical}
\label{Lem: large deviation of graph distance Dlambda}
 There exists $\rho= \rho(q_0) \geq 1$ such that for all $x \in \Z^d$ and all $t \geq \rho \|x\|_{\infty} $, 
\begin{align}
 \max \{ \pr(\rmD_{q}(0,x) \in [ t,\infty) ), \, \pr(\rmD_q ([0]_q,[x]_q) \geq t) \} \leq \rho \exp(-t/\rho).
\end{align}
\end{lem}
Accordingly, it is natural to expect $\rmT([0]_q, [n\mathbf{e}_1]_q)/n$ is close to $\rmT([0]_p, [n\mathbf{e}_1]_p)/n$. In fact, it was shown in \cite[Lemma 2.11]{garet2017continuity} that for all $p>p_c(d)$,
\begin{align}\label{unchanged mup}
 \mu_p = \lim_{n \to \infty} \frac{\rmT([0]_q,[n\mathbf{e}_1]_q)}{n} \quad \text{ a.s. and in } L_1.
\end{align}

Next, we cite a result on the length of a geodesic in First-passage percolation.
\begin{lem} \cite[Proposition 5.8]{kesten1986aspects} \label{lemKesten} 
Assume that $G$, the edge weight distribution in generalized First-passage percolation, satisfies $G(0)<p_c(d)$. Then there exists $c=c(G)\in (0,1)$ such that for all $\ell \in\N$,
 \begin{equation}\label{Proposition: kesten}
 \P \left(\exists \, \gamma\in \kP_*(0):\,\, |\gamma|\geq \ell,\,\rmT(\gamma) \leq c \ell \right) \leq  \exp{(-\ell/c)},
 \end{equation}
 where $\kP_*(0)$ is the set of all paths starting at $0$.
 \end{lem}
 The following result gives large deviation estimates of the length of geodesics.
 \begin{lem} \label{Lem: length of geo TMK} 
Recall that $M=(\log{n})^3$, $K=n^2$. Let $p_0>p_c(d)$. There exists $C_1 = C_1(p_0,F)>0$
such that for all $p\in [p_0,1]$, $\ell \geq C_1 n$ and $x \in \Lambda_{2n}(0)$, we have 
 \begin{align*} 
\max \{ \P( \exists \, \gamma \in \sO_M(0,x;\Lambda_K): |\gamma| \geq \ell), \,\, \P( \exists \, \gamma \in \sO_M(0,x): |\gamma| \geq \ell) \} \leq C_1 \exp(- \ell/(C_1M)).
\end{align*}
\end{lem}
\begin{proof}
 We first claim that there exists $C=C(q_0)>0$ such that for all $q\geq q_0$, and $\ell \geq n$ and $x \in \Lambda_{2n}$, 
\begin{align}\label{pal}
\pr(\kE_{q}^c) \leq \exp(-\ell/ (C M)), \text{ with }\kE_{q} := \{ \exists \, u \in \Lambda_{\ell/ M}(0) \cap \kC_{q}, \exists \, v \in \Lambda_{\ell/ M}(x) \cap \kC_{q}: {\rm D}_q(u,v) \leq C \ell \}.
\end{align}
 By the monotone coupling, the function  $q\mapsto \P(\kE_q)$ is increasing. Hence, it suffices to show \eqref{pal} with $q_0$. By Lemma \ref{lem:hole}, $\P(\Lambda_{\ell/ M}(z) \cap \kC_{q_0}=\emptyset)\leq e^{-\ell/ (C' M)}$ for $z\in\{0,x\}$ with some $C'=C'(q_0)>0$. Moreover, by Lemma \ref{Lem: large deviation of graph distance Dlambda}, there exists $C''=C''(q_0)>0$,
 $$\P(\exists \, u\in \Lambda_{\ell/ M}(0) \cap \kC_{q_0},\,\exists \, v\in \Lambda_{\ell/ M}(x) \cap \kC_{q_0}:~{\rm D}_{q_0}(u,v) >C'' \ell)\leq \exp(-\ell/C''),$$
 which yields \eqref{pal}. Let $q:=\P(\tau_e\leq \lambda)$. On the event $\kE_q$, there exist $u \in \Lambda_{\ell/ M}(0) \cap \kC_{q}$ and $ v \in \Lambda_{\ell/ M}(x) \cap \kC_{q}$ such that ${\rm D}_q^{\Lambda_K} (u,v) \leq C \ell$. Hence, if $\ell \leq 4dnM$, then since $\ell \leq 4dnM =o(K)$, one has ${\rm D}_q^{\Lambda_K} (u,v)={\rm D}_q(u,v).$ Thus, 
$${\rmT_M^{\Lambda_K}(u,v) \leq \lambda \rmD_q^{\Lambda_K} (u,v)= \lambda {\rm D}_q (u,v)} \leq C \lambda \ell.$$ 
Therefore, if $\ell \leq 4d n M$ and $\kE_q$ occurs, for $n$ large enough, then
\begin{align} \label{l<lnlog^3n}
 \rmT_M^{\Lambda_K}(0,x) \leq \rmT_M^{\Lambda_K}(0,u)+\rmT_M^{\Lambda_K}(u,v)+ \rmT_M^{\Lambda_K}(v, x) \leq 2d \ell + C \lambda \ell = (2d+C \lambda) \ell.
\end{align}
If $\ell > 4d n M$, then we have the same bound since $\rmT_M^{\Lambda_K}(0,x) \leq 2d M |x|_{\infty} \leq 4d n M .$  Let $C_1:= \frac{2d+ C \lambda }{c}$ with $c=c(F_{1}^1)\in (0,1)$ as in Lemma \ref{lemKesten}.  We write $\pr_G$ for the probability measure of First-passage percolation with weight distribution $G$.  By \eqref{l<lnlog^3n}, we get for all $\ell \geq n$,
\begin{align*}
 \P( \exists \, \gamma \in \sO_M(0,x;\Lambda_K): |\gamma| \geq C_1\ell, \, \kE_q) & \leq \P( \exists \, \gamma \in \sO_M(0,x;\Lambda_K): |\gamma| \geq C_1\ell, \, \rmT_M(\gamma) \leq (2d+C \lambda) \ell) \\
 & \leq \pr_{F^M_p} ( \exists \, \gamma \in \kP_*(0);~ |\gamma| \geq C_1 \ell, \rmT (\gamma) \leq c C_1 \ell ).
\end{align*}
Also, we have the same bound for $\sO_M(0,x)$ instead of $\sO_M(0,x;\Lambda_K)$. %it holds 
%$$ \ \P( \exists \, \gamma \in \sO_M(0,x): |\gamma| \geq C_1\ell, \, \kE_q) \leq \pr_{F^M_p}( \exists \, %\gamma\in \kP_*(0):~ |\gamma| \geq C_1 \ell, \rmT (\gamma) \leq c C_1 \ell).$$
Since $F^M_p$ stochastically dominates $F_{1}^1$ for $n$ large enough and $F_{1}^1(0)=F(0)< p_c(d)$, the right-hand side is bounded from above by
\begin{align*} 
% \pr_{F^M_p} ( \exists \, \gamma \in \kP_*(0):~ |\gamma| \geq C_1 \ell, \rmT(\gamma) \leq c C_1 \ell) &\leq
\pr_{F_{1}^1}( \exists \, \gamma \in \kP_*(0):~ |\gamma| \geq C_1 \ell, \rmT(\gamma)\leq c C_1 \ell)
 &  \leq \exp(- C_1 \ell/c), 
\end{align*}
thanks to Lemma \ref{lemKesten}.
Combining this with \eqref{pal}, the result follows with $\max\{C,C_1\}$ in place of $C_1$. 
\end{proof}
\subsection{Comparison of $\rmT([0]_{q}, [n\mathbf{e}_1]_{q})$ and 
$\rmT_M^{\Lambda_K}(0,n\mathbf{e}_1)$} \label{sec:comp}
\begin{prop} \label{prop: t*-tm}
For all $p\in [p_0,1]$, there exists $A=A(d,p_0,F)>0$ such that
 \aln{\label{approximation for FPT by cut off}
 \E \left[\left|\rmT([0]_q, [n\mathbf{e}_1]_q)-\rmT_M^{\Lambda_K}(0,n\mathbf{e}_1) \right| \right] \leq  AM= A(\log n)^3.
 }
\end{prop}
Note that \eqref{Time constant cut off at M} follows by combining \eqref{unchanged mup} and \eqref{approximation for FPT by cut off}. The proof of \eqref{approximation for FPT by cut off} is divided into
 \ben{ \label{ttm}
 \E\left[ \left|\rmT ([0]_{q}, [n\mathbf{e}_1]_q) - \rmT_M ([0]_{q}, [n\mathbf{e}_1]_q)\right|\right]\leq A_1 M= A_1 (\log n)^3,
 }
 \ben{ \label{tmtmk}
 \E \left[\left|\rmT_M([0]_{q}, [n\mathbf{e}_1]_q)-\rmT_M^{\Lambda_K}(0,n\mathbf{e}_1)\right|\right] \leq A_2M = A_2(\log n)^3,
 }
 for some positive constants $A_1=A_1(d,p_0,F)$ and $A_2=A_2(d,p_0)$.
 {
\begin{proof}[Proof of \eqref{ttm}]
Recall that $q=pF([0,\lambda])\leq p$ and an edge $e$ is $q$-open if and only if $\tau_e\leq \lambda$. Thus, 
\ben{ \label{tldq}
\max \left\{\rmT ([0]_{q}, [n\mathbf{e}_1]_q ), \rmT_M ([0]_{q}, [n\mathbf{e}_1]_q) \right\}\leq \lambda \rmD_q ([0]_{q}, [n\mathbf{e}_1]_q).
}
 Let $\gamma_M$ be a geodesic of $\rmT_M ([0]_{q}, [n\mathbf{e}_1]_q)$, with some deterministic rules  breaking ties. Define
\begin{align*}
 \kE_n :=\kE_n^{(1)}\cap \kE_n^{(2)}:= \{ \max\{\|0-[0]_{q}\|_\infty,\|n\mathbf{e}_1-[n\mathbf{e}_1]_q\|_\infty\} \leq M\} \cap \{\forall e \in \gamma_M, \, R_e \leq (\log n)^{5/2}\}.
\end{align*}
 Let $C_1$ be a positive constant as in Lemma \ref{Lem: length of geo TMK}. Note that 
 $$(\kE_n^{(2)})^c \cap \kE_n^{(1)}  \cap \{|\gamma_M| \leq C_1 n\}\subset \{\exists \, e \in\kE(\Lambda_{2C_1 n}):~ R_e \geq (\log n)^{5/2}\}.$$
 Thus, we have
\begin{align}\label{E*^c}
 \pr(\kE^c_n) & \leq 2\pr(\|0-[0]_{q}\|_\infty > M)+ \pr(\kE_n^{(1)};~|\gamma_M| > C_1 n) + \pr(\exists \, e \in\kE(\Lambda_{2C_1 n}):~ R_e \geq (\log n)^{5/2}).
\end{align}
By Lemma \ref{lem:hole}, there exists a positive constant $c=c(q_0)$, such that
\be{
\pr(\|0-[0]_{q}\|_\infty > M) \leq \exp(-c M^{d-1}).
}
Using Lemma \ref{Lem: length of geo TMK}, we have
\begin{align}\label{gammagreater than n}
 \pr( \kE_n^{(1)};~|\gamma_M| > C_1 n) \leq C_1 (2n)^{2d} \exp(- n/M)).
\end{align}
 Finally, Proposition~\ref{Lem: effective radius} yields
\begin{align} 
 & \pr( \exists \, e \in\kE(\Lambda_{2C_1 n}):~ R_e \geq (\log n)^{5/2}) \leq C_2 n^{d} \exp(-(\log n)^{5/4}/C_2), \notag
\end{align}
with some positive constant $C_2=C_2(d,p_0)$. Putting things together, we have, with some positive constant $C=C(d,p_0)$, 
\begin{align} \label{boundE^c}
 \pr(\kE^c_n) \leq C \exp(- (\log n)^{5/4}/C).
\end{align}

 We next prove that on the event $\kE_n,$
\ben{ \label{epop}
\tau_e<M\text{ and $e$ is $p$-open}, \quad \forall \, e \in \gamma_M \setminus \kE(\Lambda_{2M}(0) \cup \Lambda_{2M}(n \eb_1)).
}
Assume $\kE_n$ and $e \in \gamma_M\setminus \kE(\Lambda_{2M}(0) \cup \Lambda_{2M}(n \eb_1))$. If $[0]_{q} \in \Lambda_{3R_e}(e)$, then one has $\rmd_{\infty}(0,e) \leq \rmd_{\infty}(0,[0]_{q})+\rmd_{\infty}([0]_{q},e) \leq M+ 3R_e <  2M-1$, which contradicts $e \in \gamma_M\setminus \kE(\Lambda_{2M}(0) \cup \Lambda_{2M}(n \eb_1))$. Thus, we have $[0]_{q} \notin \Lambda_{3R_e}(e)$. Similarly, we have $[n\mathbf{e}_1]_q \notin \Lambda_{3R_e}(e)$. Applying Proposition \ref{prop:bypass} to $\gamma = \gamma_M \in \sO_M$, we obtain a path $ \eta_e$ from $[0]_{q}$ to $ [n\mathbf{e}_1]_q$ such that $e'$ is $q$-open for all $e' \in \eta_e \setminus \gamma_M$, i.e., $\tau_e \leq \lambda$, $|\eta_e \setminus \gamma_M| \leq C_*R_e$, and $e \not \in \eta_e$. 
Thus,
\begin{align*}
 \rmT_M (\gamma_M )& \leq \rmT_M (\eta_e) = \rmT_M (\eta_e \cap \gamma_M )+ \rmT_M (\eta_e \setminus \gamma_M )\\
 &\leq \rmT_M ( \gamma_M ) -\tau_e^M + \rmT_M (\eta_e \setminus \gamma_M ) \leq \rmT_M ( \gamma_M ) -\tau_e^M + C_* \lambda R_e, 
\end{align*}
 where recall that $\tau_e^M = \tau_e \wedge M$. This yields $\tau_e^M \leq C_* \lambda R_e<M$. Thus $\tau_e<M$, and $e$ is $p$-open.

For all $x \in \Z^d$ and $N\in \N$, thanks to \cite[Theorem 7.68]{grimmett1999percolation} and Lemma \ref{hole}, $$\pr(\{\exists q\text{-crossing cluster } \kC \subset \Lambda_{3N}(x)\} \cap \{\kC_q \cap \Lambda_{N}(x) \neq \emptyset \}) \geq 1 -C\exp(-N/C),$$
for some $C =C(q_0)>0$.
Then by Lemma \ref{apqn}, $C_q$ is the unique $q$-crossing cluster of $\Lambda_{3N}(x)$ with probability at least $1- C\exp(- N/C)$, for some $C =C(d,p_0) > 0$. Using Lemma \ref{apqn} again and Lemma \ref{Lem: large deviation of graph distance Dlambda}, there exist $C=C(d,p_0)>0$ such that for all $x \in \Z^d$ and $N\in \N$,
\ben{ \label{pabnx}
\pr(\kE'_N(x))\leq C\exp(-N/C),
}
where 
\be{
\kE'_N(x) := \{ \exists \, \eta\in \mathbb{O}_p(\Lambda_{3N}(x)):~ \Diam(\eta) \geq 3N/2, \eta \cap \kC_q = \emptyset \} \cup \{\exists \, u, v \in \Lambda_{3N}(x): \rmD_q(u,v) \in [CN, \infty)\}.
}
{Here, we remark that $q$ has been chosen appropriately to apply Lemma \ref{apqn}, see \eqref{dolam} and Appendix \ref{app:lamd}.} Suppose that $\kE_n^*:=\kE_n \cap \kE'_{2M}(0)^c \cap \kE'_{2M}(n \eb_1)^c$ occurs. On the event $\kE^*_n$, $\gamma_M$ crosses the annuli $\aA_{2M}(0)$ and $\aA_{2M}(n \eb_1) $. 
Hence, by \eqref{epop}, we find two vertices $u \in \gamma_M \cap \rmA_{2M}(0) \cap \kC_q$ and $v \in \gamma_M \cap \rmA_{2M}(n \eb_1) \cap \kC_q$, such that $\rmD_q([0]_{q},u), \rmD_q([n\mathbf{e}_1]_q,v) \leq 2CM$, and $\rmT_M(u,v)=\rmT(u,v)$. % By \eqref{epop} {and \eqref{gamma1e2}, one has  $\rmT (u,v) \leq \rmT(\gamma_M^*)\leq {\rm T}_M(\gamma_M)=\rmT_M ( [0]_{q}, [n\mathbf{e}_1]_q) $}. 
\iffalse 
Hence, there exist {sub-paths $\gamma_M^*, \gamma_1, \gamma_2$ of $\gamma_M$, such that $\gamma_M^* \subset \gamma_M \setminus   \Lambda_{2M}(0) \cup \Lambda_{2M}(n \eb_1)$, and $\gamma_1 \subset \aA_{2M}(0) \cap \gamma^*_M, \, \gamma_2 \subset \aA_{2M}(n \eb_1) \cap \gamma^*_M$ and
\ben{ \label{gamma1e2}
\gamma_M^*, \gamma_1, \gamma_2 \textrm{ are $p$-open}, \, \Diam(\gamma_1), \Diam(\gamma_2) \geq 4M. 
}
} 
 By \eqref{gamma1e2}, we find two vertices $u \in \gamma^*_M \cap \rmA_{2M}(0) \cap \kC_q$ and $v \in \gamma^*_M \cap \rmA_{2M}(n \eb_1) \cap \kC_q$, such that $\rmD_q([0]_{q},u), \rmD_q([n\mathbf{e}_1]_q,v) \leq 2CM$. By \eqref{epop} {and \eqref{gamma1e2}, one has  $\rmT (u,v) \leq \rmT(\gamma_M^*)\leq {\rm T}_M(\gamma_M)=\rmT_M ( [0]_{q}, [n\mathbf{e}_1]_q) $}.  
 \fi
Hence, we have
\ba{
\rmT ([0]_{q}, [n\mathbf{e}_1]_q) &\leq \rmT ([0]_{q},u)+\rmT (u,v)+\rmT (v,[n\mathbf{e}_1]_q)\leq 4C \lambda M + \rmT_M ( [0]_{q}, [n\mathbf{e}_1]_q).
}
Combining this with $\rmT_M([0]_{q}, [n\mathbf{e}_1]_q) \leq \rmT([0]_{q}, [n\mathbf{e}_1]_q)$, we arrive at 
\ben{ \label{tpon}
|\rmT ([0]_{q}, [n\mathbf{e}_1]_q)- \rmT_M ([0]_{q}, [n\mathbf{e}_1]_q)| \1_{\kE^*_n} \leq 4 C \lambda M.
}
By \eqref{boundE^c} and \eqref{pabnx}, we have $\pr((\kE_n^*)^c) \leq C\exp(- (\log n)^{5/4}/(4C))$. By \eqref{tldq} and Lemma \ref{Lem: large deviation of graph distance Dlambda}, we have 
\begin{align*}
 \E \left[ \left|\rmT ([0]_{q},[n\mathbf{e}_1]_q) - \rmT_M ([0]_{q}, [n\mathbf{e}_1]_q)\right|1_{(\kE_n^*)^c} \right] &\leq 2 \lambda \E \left[ \rmD_q ([0]_{q}, [n\mathbf{e}_1]_q) 1_{(\kE_n^*)^c} \right] \\
 & \leq 2 \lambda \left(\E \left[ \rmD_q^2 ([0]_{q}, [n\mathbf{e}_1]_q)\right]\right)^{1/2}(\pr((\kE_n^*)^c))^{1/2} \\
 &  \leq 4 \sqrt{C} \lambda \rho n \exp(- (\log n)^{5/4}/(8C)),
\end{align*}
which converges to $0$ as $n \rightarrow \infty$. Combining the last two displays, we obtain \eqref{ttm}.
\end{proof}
\begin{proof}[Proof of \eqref{tmtmk}]
We have 
\begin{align*} 
 \E\left[\left|\rmT_M ([0]_{q}, [n\mathbf{e}_1]_q) - \rmT_M^{\Lambda_K}(0,n\mathbf{e}_1) \right| \right] 
 & \leq \E[ | \rmT_M ([0]_{q}, [n\mathbf{e}_1]_q)- \rmT_M(0,n\mathbf{e}_1)|]+\E\left[\left|\rmT_M(0,n\mathbf{e}_1) -\rmT_M^{\Lambda_K}(0,n\mathbf{e}_1)\right|\right].
\end{align*}
 % we have
%\begin{align}\label{Expectation:0-0*}
% \E\left[ \rmT_M(0,[0]_{q}) \right]= \E\left[ \rmT_M (n\mathbf{e}_1,%[n\mathbf{e}_1]_q) \right] \leq M \E [\|0-[0]_{q}\|_1] = \kO(M).
%\end{align}
By the triangular inequality, the translation invariance, and \eqref{Claim: 2e0},  the first term is bounded from above by
\al{
\E[\rmT_M (0,[0]_{q})] + \E[\rmT_M(n \mathbf{e}_1,[n\mathbf{e}_1]_q)]\leq  2dM \E[{\rm d}_\infty(0,[0]_q)] \leq C M ,
}
for some $C = C(d,q_0)$. We now estimate the last term. We take  a geodesic of $\rmT_M(0,n\mathbf{e}_1)$, denoted by $\gamma_M$, with some deterministic rules  breaking ties. If $| \gamma_M | < n^2=K$, then $\rmT_M(0,n\mathbf{e}_1) = \rmT_M^{\Lambda_K}(0,n\mathbf{e}_1)$. 
Therefore, since $|\rmT_M(0,n\mathbf{e}_1)-\rmT_M^{\Lambda_K}(0,n\mathbf{e}_1)|\leq M n$, by  Lemma \ref{Lem: length of geo TMK}, we have
 \begin{align} \label{component 2}
 \E \left[ \left|\rmT_M(0,n\mathbf{e}_1)-\rmT_M^{\Lambda_K}(0,n\mathbf{e}_1) \right| \right] &= \E \left[ \left|\rmT_M(0,n\mathbf{e}_1)-\rmT_M^{\Lambda_K}(0,n\mathbf{e}_1) \right|\,\1_{| \gamma_M|\geq n^2} \right] \notag \\
 & \leq Mn \pr( |\gamma_M | \geq n^2 )\leq C M n \exp(-n^2/(CM)),
 \end{align}
with some $C=C(p_0,F)>0$.  This  yields \eqref{tmtmk}.
 \end{proof}
\subsection{Bound on the derivative of first passage time}
\begin{prop} \label{prop:fdb}
There exists a positive constant  $C=C(d,p_0,F)$ such that for all $p\in [p_0,1)$,
 \begin{align*}
\left| \frac{{\rm d} \E \left[{\rmT_M^{\Lambda_K}(0,n\mathbf{e}_1)} \right]}{{\rm d}p} \right|\leq Cn.
 \end{align*}
\end{prop}
\begin{proof} Let $\Delta_e \rmT_M^{\Lambda_K}(0,n\mathbf{e}_1):=\rmT_{M,+,e}^{\Lambda_K}(0,n\mathbf{e}_1)-\rmT_{M,-,e}^{\Lambda_K}(0,n\mathbf{e}_1)$, where $\rmT_{M,\pm,e}^{\Lambda_K}(0,n\mathbf{e}_1)$ is the first passage time when the weight of the edge \( e \) is set to \( M \) for \( + \) and \( 0 \) for \( - \).
\iffalse
Observe that $X := \rmT_{M}^{\Lambda_K}(0,n\mathbf{e}_1)$ is increasing with respect to $\tau^M=(\tau_e \wedge M)_{e \in \kE(\Z^d)}$. Thus, for all edge $e$,
\be{
|\E[X(\tau^M)]- \E[X(\tau^{M,+,e})]| \leq \Delta_e \rmT_M^{\Lambda_K}(0,n\mathbf{e}_1):=\rmT_{M,+,e}^{\Lambda_K}(0,n\mathbf{e}_1)-\rmT_{M,-,e}^{\Lambda_K}(0,n\mathbf{e}_1).
}
\fi
  Let $\gamma$ be a geodesic of $\rmT_{M}^{\Lambda_K}(0,n\mathbf{e}_1)$, with some deterministic rules  breaking ties. Note that the family of i.i.d. truncated weights $(\tau^M_e)_{e \in \cE(\Z^d)}$ has the same distribution as 
$$G_{p} = p (F \1_{[0,M)}+ F([M,\infty)) \delta_M) + (1-p) \delta_{M},$$
where recall that $\tau_e^M = \tau_e \wedge M$. Since $\Delta_e {\rmT_M^{\Lambda_K}(0,n\mathbf{e}_1)}=0 $ for all $e \notin \gamma$ and $\rmT_M^{\Lambda_K}(0,n\mathbf{e}_1)$ is an increasing function of weights $(\tau^M_e)_{e \in \Lambda_K}$, applying Lemma \ref{lemrusso} with $L =M,E = \kE(\Lambda_K), \xi= (\tau^M_e)_{e \in \Lambda_K}, X = \rmT_{M}^{\Lambda_K}(0,n\mathbf{e}_1)$, we have 
\begin{align}\label{rufo}
\left| \frac{{\rm d} \E \left[{\rmT_M^{\Lambda_K}(0,n\mathbf{e}_1)} \right]}{{\rm d}p} \right| & \leq \E \left[ \sum_{e \in \kE(\Lambda_K)} \Delta_e {\rmT_M^{\Lambda_K}(0,n\mathbf{e}_1)}  \right] = \E \left[ \sum_{e \in \gamma} \Delta_e {\rmT_M^{\Lambda_K}(0,n\mathbf{e}_1)}  \right].
\end{align}
We give an upper bound for \eqref{rufo}. Let $(R_e)_{e \in \kE(\Z^d)}$
and $C_* $ be 
as in Proposition \ref{Lem: effective radius}. We fix $e \in \gamma$ and define $\kU_e := \{0,n\mathbf{e}_1 \not \in \Lambda_{3R_e}(e) \}$. Define also
\begin{align*}
   \kW := \{ \forall \, \eta \in \mathbb{G}_{M}(0, n \mathbf{e}_1): |\eta| < n^2\}.
\end{align*}
Notice that if the event $\kW$ occurs then $\rmT_M(0,n\mathbf{e}_1) = \rmT_M^{\Lambda_K}(0,n\mathbf{e}_1)$, and so $\gamma$ is a geodesic of $\rmT_M(0,n\mathbf{e}_1)$. Hence, on  
$\kU_e  \cap \kW $, by Proposition \ref{prop:bypass}, there exists a path $\eta_e$ from $0 $ to $n \mathbf{e}_1$ satisfying $\eta_e \setminus \gamma$ consisting of edges with weights at most $\lambda$ and $|\eta_e\setminus \gamma| \leq C_* R_e$. Thus, on $\{R_e \leq M\} \cap \kU_e \cap \kW$, one has the bound
 \be{ \label{Disrepancy: ue}
 \Delta_e {\rmT_M^{\Lambda_K}(0,n\mathbf{e}_1)} = \rmT_{M,+,e}^{\Lambda_K}(0,n\mathbf{e}_1)-\rmT_{M,-,e}^{\Lambda_K}(0,n\mathbf{e}_1) \leq \lambda |\eta_e\setminus \gamma| \leq C_* \lambda R_e.
 }
 Otherwise, we use a trivial bound
 $\Delta_e {\rmT_M^{\Lambda_K}(0,n\mathbf{e}_1)} \leq M. $ 
We note that the event $\{R_e \leq M\} \cap \kU_e^c$ implies $ {\rm d_\infty}(0,e) \wedge {\rm d_\infty}(n \mathbf{e}_1,e) \leq 3M$. Therefore,
\begin{align} \label{Bound of sum of influence}
 \sum_{e \in \gamma } \Delta_e \rmT_M^{\Lambda_K}(0,n\mathbf{e}_1)
 & \leq C_* \lambda \sum_{e \in \gamma } R_e \1_{R_e \leq M} + M \sum_{e \in \gamma } \1_{{\rm d_\infty}(0,e) \wedge {\rm d_\infty}(n \mathbf{e}_1,e) \leq 3M} + M\sum_{e \in \gamma } \1_{R_e > M} + M |\gamma| \1_{\kW^c}  \notag \\
& \leq C_* \lambda \sum_{e \in \gamma } R_e \1_{R_e \leq M} + 4dM(6M+1)^d + M\sum_{e \in \gamma } \1_{R_e > M} + M |\gamma| \1_{\kW^c} . 
\end{align}
On the other hand, thanks to Lemma \ref{Lem: length of geo TMK}, for all $\ell \geq Cn$ with $n$ large enough,
 \ben{\label{ldga}
 \pr(|\gamma| \geq \ell ) \leq \exp(-\ell/(CM)) \leq \ell^{-5},
 }
where $C=C(p_0,F)$ is a positive constant. Therefore, using Corollary \ref{corre} with $L=n$,
\ben{\label{esre}
 \E \left[ \sum_{e \in \gamma } R_e \1_{R_e \leq M}\right] \leq C'n,
 }
with some $C'=C'(d,p_0,F)>0$. In addition, by using \eqref{ldga}, Proposition \ref{Lem: effective radius} and $M=(\log n)^3$, 
 \aln{
\E \left[\sum_{e \in \gamma } \1_{R_e > M} \right]
&\leq \E \left[\sum_{e \in \gamma } \1_{R_e > M};~|\gamma|\leq Cn \right]+\E \left[|\gamma|;~|\gamma|\geq Cn \right] \notag \\
&  \leq (2Cn+1)^{d} \, \pr(\exists \, e\in \kE( \Lambda_{Cn}): R_e > M ) + \sum_{\ell \geq Cn} \ell \, \pr(|\gamma|=\ell) \leq C_1,
}
for some constant $C_1 =C_1(p_0,F)$. By   Proposition \ref{Lem: length of geo TMK},
\ben{ \label{emwc}
\E[M |\gamma| \1_{\kW^c}] \leq M (\E[|\gamma|^2])^{1/2} (\pr(\kW^c))^{1/2}
\leq C_2,
}
for some $C_2=C_2(p_0,F)$. Combining \eqref{rufo}, \eqref{Bound of sum of influence} and \eqref{esre}--\eqref{emwc}, we  yield the desired result.
\end{proof}
\subsection{Proof of Theorem \ref{Theorem: main}}
We write $\E_u$ to emphasize that the considering parameter is $u$. By Proposition \ref{prop: t*-tm} and \eqref{unchanged mup}, and Proposition \ref{prop:fdb}, there exists  a positive constant $C=C(d,p_0,F)$ such that for all $p_1,p_2\in [p_0,1]$,
%\ben{ \label{tdbn}
\al{
|\mu_{p_2}-\mu_{p_1}|&=\lim_{n\to\infty} \frac{1}{n}\left |\E_{p_2}\left[\rmT_M^{\Lambda_K}(0,n\mathbf{e}_1)\right]-\E_{p_1}\left[\rmT_M^{\Lambda_K}(0,n\mathbf{e}_1) \right] \right| \\
&=\lim_{n\to\infty} \frac{1}{n} \left|\int_{p_1}^{p_2} \frac{{\rm d} \E_u \left[{\rmT_M^{\Lambda_K}(0,n\mathbf{e}_1)} \right]}{{\rm d}u}{{\rm d}u} \right| \leq C |p_2-p_1|.
}
\iffalse
By Proposition \ref{prop: t*-tm} and \eqref{unchanged mup}, for all $p >p_c$
\be{
\mu_p= \lim_{n\rightarrow \infty} \frac{1}{n} \E [ \tilde{\rmT}_p(0,n\mathbf{e}_1)] = \lim_{n\rightarrow \infty} \frac{1}{n} \E_p[\rmT_M^{\Lambda_K}(0,n\mathbf{e}_1)].
}
%Therefore, letting $n$ tend to the infinity in \eqref{tdbn}, we get 
 $|\mu_p-\mu_q| \leq C|p-q|.$
 \fi
\qed

\appendix 

\section{ Russo's formula: Proof of Lemma \ref{lemrusso}} \label{Appendix: ruso}

\begin{proof}
 We enumerate  $E = \{e_1,e_2,\ldots,e_n\}$. For all vector $\mathbf{p}=(p_1,p_2,\ldots,p_n) \in [0,1)^n$, let $\xi^{\p}=(\xi^{\p}_{e_i})_{i\in [n]}$ be a collection of independent random variables with the distributions $(G_{p_i})_{i\in[n]}$.  Let $(U_i)_{i=1}^n$ be i.i.d. random variables uniformly distributed on $[0,1]$ and $s = (s_{e_i})_{i=1}^n$ i.i.d. random variables taking values on $[0,L]$ with the same distribution as $\nu$, which are independent from $(U_i)$. Let us define $\omega^{\p} =(\omega_{e_i}^\p)_{i=1}^{n}$ by
 \begin{align}
 \omega_{e_i}^\p := \1_{\{U_i \leq p_i\}} s_{e_i}+\1_{(U_i >p_i)} L .
 \end{align}
 It is clear that $\omega^\p$ has the same law as $\xi^{\p}$. %By this coupling, for all $i\in [n]$ and $\varepsilon >0$ small,
% \begin{align*}
% f(\p+\varepsilon \mathbf{e}_i )-f(\p) = \E[ X(\omega^{\p+\varepsilon \mathbf{e}_i })]-\E[X(\omega^\p)],
% \end{align*}
 { Given $i \in [n]$, we consider $\hat{\omega}^\p_{e_i}$ so that $\omega^\p=(\hat{\omega}^\p_{e_i},\omega^\p_{e_i})$ to emphasize $i$ is the considering coordinate}.  
 Let $\mathbf{e}_i$ be the $i^{\text{th}}$ unit vector in $\R^n$. If $U_i \notin (p_i,p_i + \varepsilon]$, then
 $X(\omega^{\p+\varepsilon \mathbf{e}_i }) = X(\omega^\p).$ 
 Otherwise, 
 $X(\omega^{\p+ \varepsilon \mathbf{e}_i}) = X(\hat{\omega}^\p_{e_i},s_{e_i})$ and $X(\omega^{\p}) = X(\hat{\omega}^\p_{e_i},L)$. 
%where  for all given number $s \in [0,L]$,  $\omega^{\p,s,e_i}$ denotes the configuration obtained from $\omega^{\p}$ by forcing the value of $\omega^{\p}_{e_i}$ equals to $s$, i.e., $\omega^{\p,s,e_i}_{e_j}
 %= \1 (j \neq i) \omega^{\p}_{e_j} + \1(j=i)s$. 
 %$ \omega^{\p,s,e_i}= (\omega^{\p,c,e_i}_{j})_{1 \leq j \leq n}$ by
 %\begin{align*}
 %\omega^{\p,c,e_i}_{j}
 %= \1 (j \neq i) \omega^{\p}_{e_j} + \1(j=i)c,
 %\end{align*}
 Therefore, by the independence of $(U_i)_{i=1}^n$ and $(s_{e_i})_{i=1}^n$, defining 
 $f(\p) := \E[X(\xi^{\p})],$
 \begin{align*}
 f(\p+ \varepsilon \mathbf{e}_i)-f(\p) & = \E \big[(X(\hat{\omega}^\p_{e_i},s_{e_i}) - X(\hat{\omega}^\p_{e_i},L))  \1_{\{U_i \in (p_i,p_i + \varepsilon]\}}\big] = \varepsilon(\E [X(\hat{\omega}^\p_{e_i},s_{e_i})] - \E[X(\hat{\omega}^\p_{e_i},L)]).
 \end{align*}
 Let $\xi^{\p,i}$ and $\xi^{\p,+,i}$ be the configurations obtained from $\xi^{\p}$ by replacing $\xi^{\p}_{e_i}$ with $s_{e_i}$ and with $L$ respectively.  Therefore, we have
 \begin{align*}
 \dfrac{\partial f(\p)}{\partial p_i}= \lim_{\varepsilon\to 0} \frac{f(\p+\varepsilon \mathbf{e}_i)-f(\p) }{\varepsilon}= \E [X(\xi^{\p,i})] - \E [ X(\xi^{\p,+,i})].
 \end{align*}
 Combining this with the chain rule, 
 $\dfrac{{\rm d} \E[X]}{{\rm d} p} = \sum_{i=1}^n \dfrac{\partial f(\p)}{\partial p_i}\Big|_{p_1 =\ldots=p_n=p}$, we get the desired result.
\end{proof}
\section{Effect of resampling: Proof of Propositions \ref{Lem: effective radius} and \ref{prop:bypass}} \label{appendix: effect of resam}
 For $m, N \in \N$, let $\kB_N(m)$ denote the set of all boxes of side length $m$ in $\Lambda_N$.

\subsection{Choice of $\lambda$ and definition of a good box} \label{app:lamd}

\noindent Given $p_c(d)<s \leq r \leq 1$ and $m, N \in \N$, we define 
\be{
 A_{r,s,m,N} := \{\exists \, s\textrm{-crossing cluster } \kC \subset \Lambda_N,\,\exists\, \gamma \in \mathbb{O}_r(\Lambda_N): \, \Diam(\gamma)\geq m/2, \, \gamma \cap \kC = \emptyset \}.
}
\begin{lem} \label{apqn}
For all $p_0 > p_c(d)$, there exist $\delta_0 = \delta_0(p_0) >0,\,  C = C(d, p_0)>0$, such that for all $r\in [p_0,1]$, $s\in [r-\delta_0,r]$, and $N\in\N$, $(\log N)^2 \leq m\leq N$,
\be{
\pr(A_{r,s,m,N}) \leq C \exp(-m/C).
}
\end{lem}
 \begin{proof}
 
 First, we consider the case $m=N$. For simplicity, we write $A_{r,s,N}$ for $A_{r,s,N,N}$. Let $q_0:=(p_0+p_c(d))/2$. By \cite[Lemma 7.104]{grimmett1999percolation}, for all $k,N \in \N$, $s \geq q_0$, we have 
 \ben{ \label{pdc}
 \pr(\exists \, \textrm{two $s$-open clusters $\kC_1$, $\kC_2 \subset \Lambda_N$}: \Diam(\kC_1), \Diam(\kC_2) \geq k, \kC_1 \cap \kC_2 =\emptyset) \leq CN^{2d} \exp(-k/C),
 }
with some $C=C(d,q_0)>0$.\footnote{Though \cite[Lemma 7.104]{grimmett1999percolation} is only stated in $d \geq 3$, the result also holds for planar percolation by standard arguments.}
 Consequently, 
 $ \pr(A_{s,s,N}) \leq C\exp(-N/(2C^2)).$

 Moreover, by standard use of Russo's formula, e.g., \cite[(3.4)]{garet2017continuity}, we have 
 \be{
 \pr(A_{r,s,N}) \leq \pr(A_{s,s,N}) \exp(N \log (1+(r-s)/r)).
 }
 Combining the last two displays, as long as $s$ is close enough to $r,$ we have the claim.

 Next, we consider a general $m$. 
 Using \cite[Theorem 7.68]{grimmett1999percolation} and the assumption $(\log N)^2 \leq m \leq N$,
 \be{
 \pr ( \textrm{$\exists \, s$-crossing cluster $\subset \Lambda_N$, $\exists\, s$-crossing cluster $\subset \Lambda$ for all $ \Lambda \in \kB_N(m)$}) \geq 1- C \exp(-m/C),}
with some $C=C(d,p_0)>0$. It follows from this estimate and \eqref{pdc} that
 \ben{ \label{eqmn}
 \pr(\kE_{s,m,N}) \geq 1-C\exp(-m/C),
 }
 where $C$ is a positive constant depending on $d,p_0$ and 
 \be{
 \kE_{s,m,N}:= \{ \exists \, \textrm{$s$-crossing cluster $\kC \subset \Lambda_N$ that contains a $s$-crossing cluster in $\Lambda$ for all $\Lambda \in \kB_N(m)$} \}.
 } 
Remark that if the event $A_{s,s,N}^c$ occurs, then there is at most one $s$-crossing cluster in $\Lambda_N$. Notice further that given a path $\gamma$ with $\Diam(\gamma) \geq m/2$ in $\Lambda_N$, there exist $x,y \in \gamma$ such that $\|x-y\|_\infty \geq m/2$. Let $\gamma_{x,y}$ be the sub-path of $\gamma$ from $x$ to $y$, so there exists a sub-path $\gamma' \subset \gamma_{x,y} \cap \Lambda_{m/2}(x)$ such that $\Diam(\gamma')\geq m/2$. Thus, we can find $\Lambda\in \kB_N(m)$ such that $\Lambda$ contains $\Lambda_{m/2}(x)\cap \Lambda_N$, and so $\gamma' \subset \Lambda$ with $\Diam(\gamma')\geq m/2$. Therefore, if $A_{r,s,m,N} \cap A_{s,s,N}^c \cap \kE_{s,m,N}$ occurs, then there exists a box $\Lambda \in \kB_N(m)$, a $r$-open path $\gamma' \in \O_r( \Lambda)$ with $\Diam(\gamma')\geq m/2$ and a $s$-crossing cluster $\kC' \subset \Lambda$ such that $\gamma' \cap \kC' =\emptyset$. Hence, using the claim for $A_{r,s,m}$ and $ (\log N)^2 \leq m \leq N$,
 \be{
 \pr(A_{r,s,m,N} \cap A_{s,s,N}^c \cap \kE_{s,m,N}) \leq |\kB_N(m)| \pr(A_{r,s,m}) \leq C\exp(-m/C),
 }
 with $C=C(d,p_0)$ a positive constant. Combining all together gives the desired result. 
 \end{proof} 
 With  $\delta_0$ as in Lemma \ref{apqn}, we recall that $q_0 = \tfrac{p_0 +p_c(d)}{2}$ and   $\lambda=\lambda(p_0,F)$ defined in \eqref{dolam}, such that
$$
 \quad q=pF([0, \lambda]) \geq q_0,\, q_0 \leq q \leq p \leq q + \delta_0, \, \, \forall \,  p\in [p_0,1].
$$
Hence, these parameters $p$ and $q$ satisfy the condition of Lemma \ref{apqn}, allowing us to utilize this lemma in the following result.
%Let $\delta_0$ is a positive constant as in Lemma \ref{apqn}. Given $p_0 > p_c(d)$, letting $q_0:=\frac{p_0+p_c(d)}{2}$, we pick a sufficiently large constant $\lambda=\lambda(F,p_0)$ and set $q:= \pr(\tau \leq \lambda)
% =p F([0, \lambda])$ such that
%\ben{ \label{lambd}
%F([0,\lambda]) \geq \max\left\{ \frac{q_0}{p_0},1-\delta_0\right \},\quad q_0 \leq q \leq p \leq q + \delta_0 \, \, \forall \, p \in [p_0,1].
%}
%From now on, we fix such $\lambda$ and $q$. 
\begin{lem} \label{lem: crossing cluster}
 There exists $C=C(d,p_0) \geq 3$ such that for all $t \geq C$, $H >0$ and $N \in [ H^2/C]$,
 \be{
\pr(\exists \, q\textrm{-crossing cluster } \kC \subset \Lambda_N,\,\exists \, \pi \in \kP(\Lambda_N) \cap \sO_{H}(\Lambda_{tN}): \Diam(\pi)\geq N/2, \, \pi \cap \kC = \emptyset ) \leq C\exp(-\sqrt{N}/C).
} 
\end{lem}
\begin{proof} 
Using Lemma \ref{apqn}, there exists $C_1 =C_1(d,p_0) > 0$ such that for all $N \geq 1$,
\be{ \label{p-qest}
\pr(\exists \, q\textrm{-crossing cluster } \kC \subset \Lambda_N,\,\exists \, \pi \in \mathbb{O}_p (\Lambda_N): \text{Diam}(\pi)\geq \sqrt{N}/2, \, \pi \cap \kC = \emptyset) \leq C_1\exp(-\sqrt{N}/C_1).
}
Hence, the result follows if there exists $C_2=C_2(q_0)>0$ such that for all $N \leq H^2/C_2$,
\begin{align}\label{epq}
\pr & (\forall \, \pi \in \kP(\Lambda_N) \cap \sO_{H}(\Lambda_{tN})\text{ with }\Diam(\pi) \geq N/2,\, \exists \,\eta \subset \pi: \eta \in \mathbb{O}_p(\Lambda_N)\text{ and } \Diam(\eta) \geq \sqrt{N}/2) \notag \\
& \geq 1- C_2 \exp(-\sqrt{N}/C_2).
\end{align} 
Let $\cl_p(\pi)$ denote the set of $p$-closed edges of $\pi$. Observe that if $\Diam(\pi) \geq N/2$ and $|\cl_p(\pi)| \leq \sqrt{N}/2$, then $\pi$ contains a $p$-open sub-path, say $\eta$, with $\Diam(\eta) \geq \sqrt{N}/2$. 
Moreover, if $\pi=(x, \ldots,y) \in \sO_H(x,y;\Lambda_{tN})$ satisfies $|\cl_p(\pi)| \geq \sqrt{N}/2$, then $\rmT_H^{\Lambda_{tN}}(x,y) = \rmT_H(\pi) \geq \sqrt{N}H/2$. Hence, it suffices to show
\ben{ \label{tmxy}
\pr(\exists \, x,y \in \Lambda_N: \rmT^{\Lambda_{tN}}_H(x,y) \geq \sqrt{N} H/2) \leq C_2\exp(-\sqrt{N}/C_2),
}
with some $C_2=C_2(q_0)>0.$ By Lemmas \ref{lem:hole} and \ref{Lem: large deviation of graph distance Dlambda}, there exists $C_3=C_3(q_0)>0$ such that
$$\begin{array}{lll}
&\pr(\kA_N)\leq C_3\exp(-\sqrt{N}/C_3), & \kA_N
:=\{\exists \, x \in \Lambda_N: \rmd_1(x,[x]_q) \geq \sqrt{N}/8\},\\ 
& \pr(\kB_
N) \leq C_3\exp(-N/C_3), &\kB_N
:=\{\exists \, u, \, v \in \Lambda_{2N} \cap \kC_q: \rmD_q(u,v) \geq C_3 N \}.
\end{array}$$
Given $x,y \in \Lambda_N$, let $\eta_x$ (resp. $\eta_y$) be a path with the shortest length in $\Z^d$-lattice from $x$ to $[x]_{q}$ (resp. from $y$ to $[y]_{q}$), and $\eta_{x,y}$ a geodesic of $\rmD_q([x]_{q}, [y]_{q})$. Construct a path from $x$ to $y$ by $\eta:=\eta_x\cup \eta_{x,y} \cup \eta_y$. On the event, $\kA_N^c \cap \kB_N^c$, for all $t \geq 2C_3$ and $x, y \in \Lambda_N$, since $\eta\in \kP(\Lambda_{2C_3N}),$
\be{
\rmT_H^{\Lambda_{tN}}(x,y) \leq \rmT_H(\eta_x) + \rmT_H(\eta_{x,y}) + {\rmT_H(\eta_y)} \leq H [\rmd_1(x,[x]_{q})+ \rmd_1(y,[y]_{q})] + \lambda \rmD_q([x]_{q},[y]_{q}) < \sqrt{ N} H/2,
}
provided that $N \leq H^2/(8C_3 \lambda)^2$. Hence, \eqref{tmxy} follows. 
\end{proof}

Recall ${\rm A}_N(e)= \Lambda_{3N}(e) \setminus \Lambda_N(e)$.
Fix $\rho$ and $C(d,p_0)$ as in Lemma \ref{Lem: large deviation of graph distance Dlambda},  \ref{lem: crossing cluster}, respectively and set 
\aln{\label{def of C-star}
N_{\rho}:= \left \lfloor N/8 \rho^2 \right \rfloor, \quad C_*:=C(d,p_0) + (48\rho
^2)^d.
}
\begin{defin}\label{Def: good annulus}
For each $e \in \kE(\Z^d)$, we say that the box $ \Lambda_{3N}(e)$ is $q$-\textbf{good} if the following hold:
\begin{itemize}
 \item [(i)] There exists a $q$-crossing cluster $\kC$ in $\Lambda_{3N}(e)$ that contains a crossing cluster in $\Lambda$ for all $\Lambda\in \kB_{3N}(N_{\rho})$,
 \item [(ii)] For all $x,y \in \aAn$ with ${\rm d}_{\infty} (\{x,y\},\partial \aAn) \geq N/2$ and ${\rm d}_{\infty} (x,y) \leq 2N_\rho$, if $\rmD_q(x,y) < \infty$, then $\rmD_q^{\aAn}(x,y) = \rmD_q(x,y) \leq 4\rho N_\rho$.
 \item [(iii)] If $\pi \in \kP(\Lambda_{3N}(e)) \cap \sO_H(\Lambda_{C_*N}(e))$ satisfies $\Diam(\pi) \geq N_\rho$, then $\pi \cap \kC \neq \emptyset$.
\end{itemize}
\end{defin}

\begin{lem} \label{Lem: being bad annulus}
 There exists  $C=C(d,p_0)>0$ such that for all $e \in \kE(\Z^d)$, $H >0$, and $N \in [ H^2/C]$
 \begin{align*}
 \pr(\Lambda_{3N}(e) \text{ is } q\textbf{-good} ) \geq 1- C\exp(-\sqrt{ N}/C).
 \end{align*}
\end{lem}
\begin{proof}
 Fix $e \in \cE(\Z^d)$. Using \eqref{eqmn}, there exists a positive constant $C=C(d,p_0)$, such that
$$\pr(\textrm{$\Lambda_{3N}(e)$ does not satisfies (i)}) \leq \pr(\kE^c_{q, N_\rho, 3N}) \leq C\exp(-N/C).$$
Observe that if $\Lambda_{3N}(e) \text{ does not satisfy (ii)},$ then there exist $x,y \in \aAn$ such that ${\rm d}_{\infty} (\{x,y\},\partial \aAn) \geq N/2$, ${\rm d}_{\infty} (x,y) \leq 2N_\rho$, $\rmD_q(x,y) \in [4\rho N_\rho,\infty).$ Hence, thanks to the union bound and Lemma \ref{Lem: large deviation of graph distance Dlambda}, there exists a positive constant $C=C(d,p_0,\rho)>16\rho^2$ such that
\begin{align} \label{(ii)}
 \pr(\Lambda_{3N}(e) \text{ does not satisfy (ii)}) \leq C |\aAn|^2 \exp(-N_\rho/C)\leq C \exp(-N/(C^2)).
\end{align}
 Suppose now that $\Lambda_{3N}(e)$ satisfies (i) but not (iii). Then, there exist $\pi \in \kP( \Lambda_{3N}(e)) \cap \sO_H(\Lambda_{C_*N}(e))$ and a $q$-crossing cluster $\kC\subset \Lambda_{3N}(e)$ such that $\Diam(\pi) \geq N_\rho$, and $\kC$ crosses all $\Lambda\in \kB_{3N}(N_{\rho})$, and $\pi \cap \kC = \emptyset$. Note that there exists a vertex $x\in \Lambda_{3N}(e)$ and a sub-path {$\pi' \in \kP(\Lambda_{N_\rho/2}(x))\cap \sO_H(\Lambda_{C_*N_\rho/2}(x))$ of $\pi$}  such that $\Diam(\pi') \geq N_{\rho}/2 $ and $\pi' \cap \kC = \emptyset$. Thus, by Lemma \ref{lem: crossing cluster}, there exists $C=C(d,p_0,\rho) > 0$ such that 
\begin{align*}
 & \pr (\Lambda_{3N}(e) \text{ satisfies (i) but not (iii)}) \\
 & \leq \pr \left(\begin{array}{c}
 \exists \, x \in \Lambda_{3N}(e), \,\exists \, \textrm{$q$-crossing cluster $\kC' \subset \Lambda_{N_{\rho}/2}(x)$},
 \exists \, \pi' \in \kP(\Lambda_{N_{\rho}/2}(x))\cap \sO_H(\Lambda_{C_* N_{\rho}/2}(x)):\\
 \Diam(\pi') \geq N_\rho/2, \, \pi' \cap \kC' = \emptyset
 \end{array}
 \right) \\
 & \leq C N^{d}\exp(- \sqrt{ N}/C).
\end{align*}
Putting things together, we have the claim. 
\end{proof}
\subsection{Proof of Proposition \ref{Lem: effective radius}}
 Recall $\rho$, $N_{\rho}$, and $C_*$  from Lemma \ref{Lem: large deviation of graph distance Dlambda} and \eqref{def of C-star}. Let 
\be{
\mathcal{V}_N(e):=\{\forall \, \gamma_1, \gamma_2 \in \sO_{H}(\Lambda_{C_*N}(e)) \cap \sC(\aAn),\, 
 {\rm D}_{q}^{\aAn}(\gamma_1,\gamma_2) \leq C_*N\}.
 } Fix $e \in \kE(\Z^d)$. By the definition of $R_e$ and Lemma \ref{Lem: being bad annulus}, the result follows from
\begin{align} \label{implies E1}
 \{ \Lambda_{3N}(e) \text{ is } q\textbf{-good} \} \subset \mathcal{V}_N(e).
\end{align}
To this end, we assume that $ \Lambda_{3N}(e) \text{ is } q\textbf{-good}$. Let $\gamma_1, \gamma_2 \in \sO_{H}(\Lambda_{C_*N}(e)) \cap \sC(\aAn)$. For each $j \in \{1,2\}$, there exists a connected path $\pi_j \subset \gamma_j \cap \left\{ \Lambda_{2N+ \frac{N_\rho}{2}}(e) \setminus \Lambda_{2N -\frac{N_\rho}{2}}(e) \right\}$ satisfying 
\begin{align*}
 \forall j \in \{1,2\}, \quad \pi_j \in \kP(\Lambda_{3N}(e)) \cap \sO_H(\Lambda_{C_*N}(e)), \quad \text{diam}(\pi_j) \geq N_\rho, \quad \rmd_{\infty}(\pi_j, \partial \aAn) \geq 3N/4.
\end{align*}
 Then, by Definition \ref{Def: good annulus} (iii), we have $\pi_1 \cap \kC \neq \emptyset$ and $\pi_2 \cap \kC \neq \emptyset$, with $\kC$ the cluster crossing all sub-boxes of side-length $N_{\rho}$ of $\Lambda_{3N}$. Therefore, there exist $ u,v \in \aAn$ such that
 $ u \in \pi_1 \cap \kC$, $ v \in \pi_2 \cap \kC$, and ${\rm d}_{\infty} (\{u,v\}, \partial \aAn) \geq 3N/4$. 
 Moreover, since $\kC$ contains a crossing cluster in $\Lambda$ for all $\Lambda \in \kB_{3N}(N_\rho)$, we find a sequence of vertices $(x_i)_{i=0}^h \subset \kC$ with $h \leq (6N/N_\rho)^d=(48 \rho^2)^d$ such that
 \be{
 x_0=u, \, \, x_h=v; \quad \rmd_{\infty}(x_i,\partial \aAn) \geq N/2 ; \quad \rmd_{\infty}(x_{i-1},x_i) \leq 2 N_\rho \quad \forall \,  i \in [h].
 }
Remark further that $\rmD_q(x_{i-1},x_i)<\infty$, as $(x_i)_{i=0}^h \subset \kC$. Hence, it follows from Definition \ref{Def: good annulus} (ii) that $\rmD_q^{\aA_n(e)}(x_{i-1},x_i) \leq 4\rho N_\rho$. Therefore, $\kV_N(e)$ holds since
\be{
{\rm D}_{q}^{\aAn}(\gamma_1,\gamma_2) \leq \sum_{i=1}^{h} \rmD_{q}^{\aAn}(x_{i-1},x_{i}) \leq (6N/N_\rho)^d (4 \rho N_\rho) \leq C_*N. 
}
\qed
\subsection{Proof of Proposition \ref{prop:bypass}}
 Assume that $\gamma=(x_i)_{i=1}^\ell \in \sO_H$ is a path between $x$ and $y$ with $x,y \in \Z^d$. If $ e \in \gamma$ and $x,y \notin \Lambda_{3R_e}(e)$, then $\gamma$ crosses the annulus $\rmA_{R_e}(e)$ at least twice. The first and last sub-path of $\gamma$ crossing $\rmA_{R_e}(e)$ are defined by $\gamma_{1}=(x_{i_-},\ldots,x_{i_+})$ and $\gamma_{2}=(x_{o_-},\ldots,x_{o_+})$, where 
\al{
i_+ &:= \min\{i\geq 1: x_i \in \partial \Lambda_N\}, \quad i_- := \max\{i\leq i_+: x_i \in \partial \Lambda_{3N}\},\\
 o_- &:= \max\{i\geq 1: x_i \in \partial \Lambda_N\}, \quad o_+ := \min\{i\geq o_-: x_i \in \partial \Lambda_{3N}\}.
}
\begin{figure}[h]
  \centering
  \includegraphics[width=0.8\linewidth]{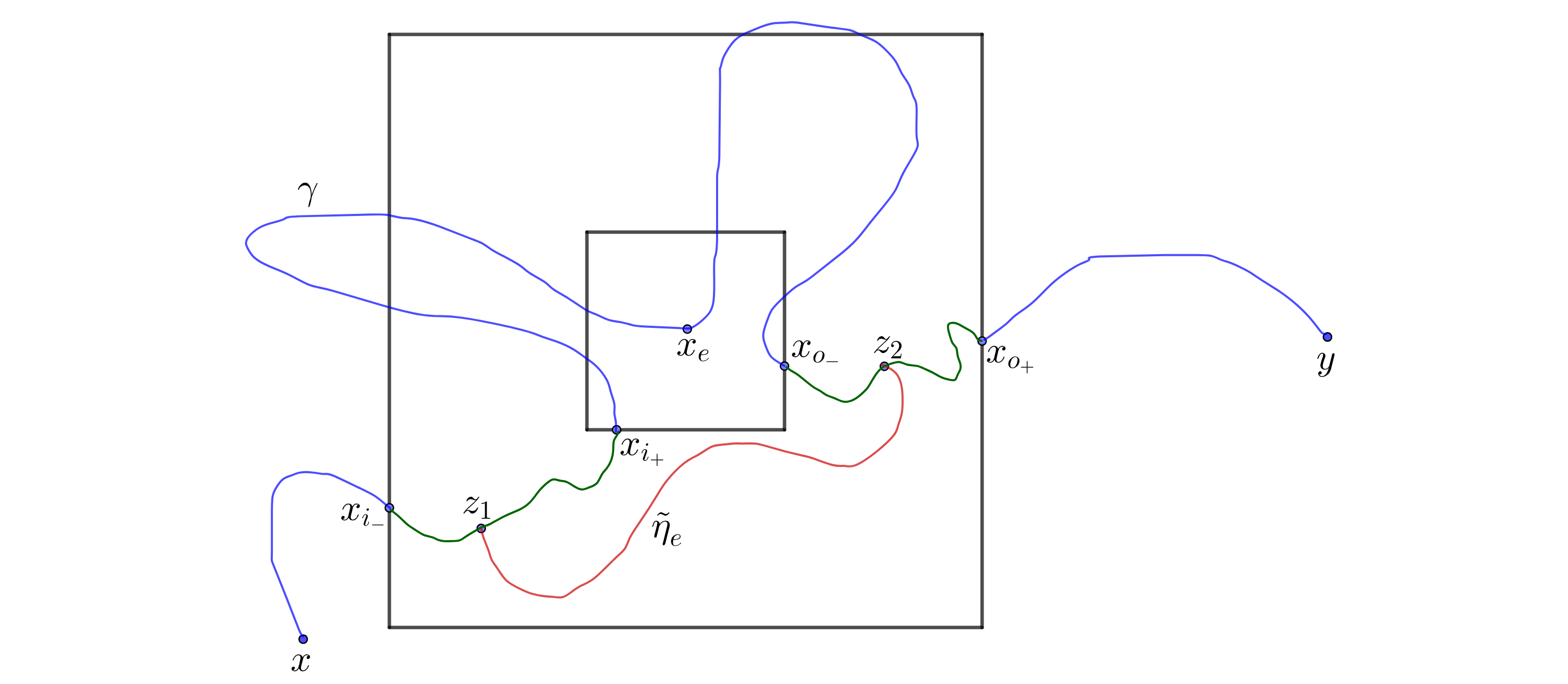}
  \caption{Construction of $\eta_e$: the green lines represent the first and last crossing paths $\gamma_1, \gamma_2$ of $\rmA_{R_e}(e)$; the red line represents the bypass $\tilde{\eta}_e$ of $e$.} 
  %$\eta_e= \gamma_{x,z_1} \cup \tilde{\eta}_e \cup \gamma_{z_2,y}$, and $\tilde{\eta}_e = \eta_e \setminus \gamma$}
  \label{fig2}
\end{figure}
We have
 $\gamma_1, \gamma_2 \in \sO_H$ and $\gamma_1, \gamma_2 \subset \rmA_{R_e}(e) \subset \Lambda_{C_*R_e}(e)$, which implies $\gamma_1,\gamma_2\in \sC(\rmA_{R_e}(e)) \cap \sO_H(\Lambda_{C_* R_e}(e))$. By definition of $R_e$, ${\rm D}_q^{\rmA_{R_e}(e)}(\gamma_1,\gamma_2) \leq C_*R_e$. Let $\tilde{\eta}_e$ be a geodesic of ${\rm D}_q^{\rmA_{R_e}(e)}(\gamma_1,\gamma_2)$. Then, it is a $q$-open path $\tilde{\eta}_e$ such that
 $|\tilde{\eta}_e| = {\rm D}_{q}^{\rmA_{R_e}(e)}(\gamma_1,\gamma_2) \leq C_* R_e.$ 
For $u,v \in \gamma$, we write $\gamma_{u,v}$ for the sub-path of $\gamma$ from $u$ to $v$. Let \( z_1 \) and \( z_2 \) be points where the path \( \tilde{\eta}_e \) intersects with \( \gamma_1 \) and \( \gamma_2 \), respectively. We define
$$\eta_e:= \gamma_{x,z_1} \cup \tilde{\eta}_e \cup \gamma_{z_2,y}.$$
 Notice that $|\eta_e \setminus \eta| =|\tilde{\eta}_e| \leq C_* R_e$. Furthermore, since $\gamma_1$ and $\gamma_2$ are first and last sub-path of $\gamma $ crossing $\rmA_{R_e}(e)$, one has $ \gamma_{x,z_1} \cap \Lambda_{R_e-1}(e)=\emptyset$ and $ \gamma_{z_2,y} \cap \Lambda_{R_e-1}(e) =\emptyset$. In addition, 
$\tilde{\eta}_e \cap \Lambda_{R_e-1}(e) = \emptyset $ since $\tilde{\eta}_e \subset \rmA_{R_e}(e)$. Hence, $\eta_e \cap \Lambda_{R_e-1}(e) = \emptyset $. Hence, $\eta_e$ is a desired path (see Figure \ref{fig2}).
\qed
\section{The strong convergence to time constant: Proof of Theorem \ref{Theorem: time constant}} \label{app:slln}
 Theorem \ref{Theorem: time constant} directly follows from Kingman's sub-additive ergodic theorem, e.g., \cite[Theorem 2.2]{auffinger201750}, assuming the following integrability of passage time recalling that $\tilde{\rmT}(x,y):=\rmT([x]_p,[y]_p)$.
 \begin{lem}
 If  $\E[\tau_e \1_{\tau_e < \infty}]< \infty$ and $p>p_c(d)$, then $\E[\rmT([0]_p,[\bef_1]_p)]< \infty$.
 \end{lem}
 \begin{proof} 
Define $X:=\inf \{m: \rmD^{\Lambda_m}_p([0]_p,[\bef_1]_p)<\infty\}$. If $X=k$, then there exists a $p$-open path in $\Lambda_k$ between $[0]_p$ and $[\bef_1]_p$, and thus $\rmTp(0,\bef_1) \leq \sum_{e \in \Lambda_k} \tau_e\1_{\tau_e < \infty}$. Let $\kE_k := \{X \geq k\}= \{ \rmD^{\Lambda_{k-1}}_p([0]_p,[\bef_1]_p)=\infty \}.$ Hence,
 \ben{ \label{etoe}
 \E[\rmT([0]_p,[\bef_1]_p)] \leq \sum_{k=1}^{\infty} \E \left[ \sum_{e \in \Lambda_k} \tau_e\1_{\tau_e < \infty} \1_{X=k} \right] \leq \sum_{k=1}^{\infty} \E \left[ \sum_{e \in \Lambda_k} \tau_e\1_{\tau_e < \infty} \1_{\kE_k} \right].
 }
 Since the event $\kE_k$ is measurable with $(\1_{\tau_e<\infty})_{e \in \kE(\Z^d)}$, we have 
 \bea{
 \E \left[ \tau_e\1_{\tau_e < \infty}\1_{\kE_k} \right] = \E \left[\tau_e\1_{\tau_e < \infty} \E\left[ \1_{\kE_k}\1_{\tau_e < \infty} \mid \tau_e \right] \right]\leq \E \left( \tau_e\1_{\tau_e < \infty}\right) \pr(\kE_k)/\pr(\tau_e < \infty).
 }
 By Lemma \ref{lem:hole} and Lemma \ref{Lem: large deviation of graph distance Dlambda}, there exists a positive constant $c$, such that 
 \bea{
 \pr(\kE_k) \leq \pr(\{[0]_p,[\bef_1]_p\} \not \subset \Lambda_{ck}) + \pr(\exists \, u,v \in \Lambda_{ck}: \rmD_p(u,v) \in (k/2, \infty))\leq c^{-1}\exp(-ck).
 }
 Combining this with \eqref{etoe} yields that 
 \be{
 \E[\rmT([0]_p,[\bef_1]_p)] \leq \sum_{k=1}^{\infty} (2k+1)^d (pc)^{-1}\exp(-ck) \E \left[ \tau_e\1_{\tau_e < \infty}\right] < \infty.
 }
 \end{proof}
\section*{Acknowledgements}
The authors thank Barbara Dembin for helpful discussions and suggestions for references.  V. H. Can is supported by the International Center for Research and Postgraduate Training in Mathematics, Institute of Mathematics, Vietnam Academy of Science and Technology under grant number $\textrm{ICRTM01}_{-}2024.01.$ V. Q. Nguyen was funded by the
Master, PhD Scholarship Programme of Vingroup Innovation Foundation (VINIF), code VINIF.2023.TS.096. S. Nakajima is supported by JSPS KAKENHI 22K20344.

\bibliographystyle{alpha}
\bibliography{citation} 
\end{document}